\documentclass[12pt]{amsart}
 \usepackage{amssymb,amsmath,amsfonts,epsfig,latexsym}

 \newtheorem{theorem}{Theorem}[section]
\newtheorem{definition}[theorem]{Definition}
\newtheorem{proposition}[theorem]{Proposition}
\newtheorem{lemma}[theorem]{Lemma}
\newtheorem{corollary}[theorem]{Corollary}

\newtheorem*{Stepanov}{Stepanov's Lemma}
\newtheorem*{PD}{Poincar\'e Duality}
\newtheorem*{pairs}{Exact Sequence of a Pair}
\newtheorem*{MV}{Mayer-Vietoris Sequence}

\theoremstyle{definition}
\newtheorem{remark}[theorem]{Remark}
\newtheorem{example}[theorem]{Example}

\def\N{\ensuremath{\mathbb{N}}}
\def\Z{\ensuremath{\mathbb{Z}}}
\def\Q{\ensuremath{\mathbb{Q}}}
\def\P{\ensuremath{\mathbb{P}}}
\def\C{\ensuremath{\mathbb{C}}}

\def\cP{\ensuremath{\mathcal{P}}}

\def\R{\ensuremath{\mathbb{R}}}

\def\<{\ensuremath{\langle}}
\def\>{\ensuremath{\rangle}}

\DeclareMathOperator{\Bl}{Bl}
\DeclareMathOperator{\coker}{coker}
\DeclareMathOperator{\conv}{conv}
\DeclareMathOperator{\discr}{discr}

\DeclareMathOperator{\Gr}{Gr}

\def\sing{\mathrm{sing}}

\DeclareMathOperator{\im}{im}

\begin{document}

\title{Boundary complexes and weight filtrations}

\author{Sam Payne}
\email{sam.payne@yale.edu}

\begin{abstract}
We study the dual complexes of boundary divisors in log resolutions of compactifications of algebraic varieties and show that the homotopy types of these complexes are independent of all choices.  Inspired by recent developments in nonarchimedean geometry, we consider relations between these boundary complexes and weight filtrations on singular cohomology and cohomology with compact supports, and give applications to dual complexes of resolutions of isolated singularities that generalize results of Stepanov and Thuillier.
\end{abstract}

\maketitle

\vspace{-30 pt}

\tableofcontents

\vspace{-30 pt}

\section{Introduction}

Let $D$ be a divisor with simple normal crossings on an algebraic variety.  The \emph{dual complex} $\Delta(D)$ is a triangulated topological space, or $\Delta$-complex, whose $k$-dimensional simplices correspond to the irreducible components of intersections of $k+1$ distinct components of $D$, with inclusions of faces corresponding to inclusions of subvarieties; see Section~\ref{dual complexes} for further details.  This paper studies the geometry and topology of dual complexes for boundary divisors of suitable compactifications, and relations to Deligne's weight filtrations.

Let $X$ be an algebraic variety of dimension $n$ over the complex numbers.  By theorems of Nagata \cite{Nagata62} and Hironaka \cite{Hironaka64}, there is a compact variety $\overline X$ containing $X$ as a dense open subvariety, and a resolution
\[
\varphi: X' \rightarrow \overline X,
\]
which is a proper birational morphism from a smooth variety that is an isomorphism over the smooth locus in $X$, such that the boundary
\[
\partial X' = X' \smallsetminus \varphi^{-1}(X),
\]
and the union $\varphi^{-1}(X^\sing) \cup \partial X'$ are divisors with simple normal crossings.  We define the \emph{boundary complex} of a resolution of a compactification, as above, to be the dual complex $\Delta(\partial X')$ of the boundary divisor.

The intersections of irreducible components of boundary and exceptional divisors, and the inclusions among them, encode a simplicial resolution of the pair $(\overline X, \overline X \smallsetminus X)$ by smooth complete varieties, and this data determines the weight filtration, and even the full mixed Hodge structure, on the cohomology of $X$.  The combinatorial data in the boundary complex captures one piece of the weight filtration; there is a natural isomorphism from the reduced homology of the boundary complex to the $2n$th graded piece of the weight filtration on the cohomology of $X$,
\begin{equation} \label{boundary homology display}
\widetilde H_{i-1}(\Delta (\partial X'); \Q) \cong \Gr_{2n}^W H^{2n-i}(X).
\end{equation}
This isomorphism has been known to experts in mixed Hodge theory, and was highlighted in a recent paper by Hacking in the case where $X$ is smooth \cite[Theorem~3.1]{Hacking08}.  See Theorem~\ref{boundary homology} for the general case.   The existence of such an isomorphism suggests that the topology of the boundary complex may be of particular interest.

Any resolution of a compactification, as above, may be seen as a compactification of a resolution.  However, not every compactification of a resolution occurs in this way.  We consider, more generally, boundary complexes for compactifications of \emph{weak resolutions} that may or may not be isomorphisms over the smooth locus of $X$.  Let
\[
\pi: \widetilde X \rightarrow X
\]
be a proper birational morphism from a smooth variety, and let $X^+$ be a smooth compactification of $\widetilde X$ such that the boundary
\[
\partial X^+ = X^+ \smallsetminus \widetilde X
\]
is a divisor with simple normal crossings.

\begin{theorem} \label{only boundary}
The simple homotopy type of the boundary complex $\Delta(\partial X^+)$ is independent of the choices of resolution and compactification.
\end{theorem}

\noindent  In the special case where $X$ is smooth, the invariance of the ordinary homotopy type of $\Delta(\partial X^+)$ is due to Thuillier, who gave a proof over perfect fields, when such compactifications exist, using nonarchimedean analytic geometry \cite{Thuillier07}.  The simple homotopy type is a finer invariant, in general, than ordinary homotopy type for regular CW complexes.  For instance, the 3-dimensional lens spaces $L_{7,1}$ and $L_{7,2}$ are homotopy equivalent but not simple homotopy equivalent \cite[Chapter~V]{Cohen73}.  The distinction for boundary complexes is not vacuous, because the simple homotopy type of any regular CW complex is realized by a boundary complex.  See Example~\ref{ex:arbitrarycx}.  

It follows from these homotopy invariance results that invariants of the boundary complex such as homotopy groups and generalized cohomology rings are also invariants of $X$.  Some of these have been known and been studied in other contexts.  For instance, the integral homology groups of the boundary complex can be computed from the motivic weight complexes of Gillet and Soul\'e \cite{GilletSoule96} and Guill\'en and Navarro Aznar \cite{GuillenNavarroAznar02}.  Others, such as the fundamental group of the boundary complex and simple homotopy type, appear to be new and interesting.  

\begin{remark}
As noted above, the reduced homology of the boundary complex is identified with the $2n$th graded piece of the weight filtration on the cohomology of $X$.  At the other end of the weight filtration, Berkovich showed that the rational cohomology of the nonarchimedean analytification $X^{\mathrm{an}}$ of $X$ with respect to the trivial valuation on the complex numbers is naturally identified with $W_0 H^*(X,\Q)$ \cite{Berkovich00}, and Hrushovski and Loeser have announced a proof that this nonarchimedean analytification has the homotopy type of a finite simplicial complex \cite{HrushovskiLoeser10}.  Therefore, there are canonical homotopy types of finite simplicial complexes associated to the first and last graded pieces of the weight filtration.
\end{remark}

\begin{remark}
The boundary complex can be embedded in $X^{\mathrm{an}}$ or, more precisely, in the ``punctured tubular neighborhood at infinity" that is written $X^\mathrm{an} \smallsetminus X^\beth$ in the notation of \cite{Thuillier07}, by mapping each vertex to a scalar multiple of the valuation on the function field $\C(X)$ given by order of vanishing along the corresponding divisor, and mapping each $k$-dimensional face to the space of suitably normalized monomial valuations in the $k+1$ local coordinates cutting out the irreducible components of the boundary that meet along the corresponding subvariety.  The closure in $X^\mathrm{an}$ of the union of the images of all such embeddings, over all suitable resolutions of compactifications of $X$, may then be seen as an intrinsic boundary complex for $X$.  Favre and Jonsson have studied two examples of these intrinsic boundary complexes in detail, with important applications to complex dynamics.  For the complement of a point in the projective plane, the intrinsic boundary complex is their valuative tree \cite{FavreJonsson04}, and for the complement of a line it is their tree of valuations at infinity \cite{FavreJonsson07, FavreJonsson11}.  The analogue of the valuative tree in higher dimensions is studied in \cite{BFJ08}.
\end{remark}

These boundary complexes can also be applied to study invariants of singularities and, more generally, invariants of pairs.  Suppose $Y$ is a closed subset of $X$, and consider a proper birational morphism from a smooth variety such that the preimage $E$ of $Y$ is a divisor with simple normal crossings.  If $X$ is proper, then the \emph{resolution complex} $\Delta(E)$ is a boundary complex for $X \smallsetminus Y$, and therefore Theorem~\ref{only boundary} says that its simple homotopy type is independent of the choice of resolution.  If one studies such resolution complexes in the case where $X$ is not compact, filtered complexes naturally appear as boundary complexes of descending chains of open subsets of varieties.  For instance, one could study a boundary complex for $X$ as a subcomplex of a boundary complex for $X \smallsetminus Y$.
In this context, it is natural to consider chains of closed subsets as well.

Let $Y_1 \subset \cdots \subset Y_s$ be a chain of closed algebraic subsets of $X$, and let $\pi$ be a weak log resolution of $X$ with respect to $Y_1, \ldots, Y_s$.  By this, we mean that $\pi$ is a proper birational morphism from a smooth variety $\widetilde X$ to $X$ such that the preimage $E_i$ of each $Y_i$ is a divisor with simple normal crossings.   Let $X^+$ be a smooth compactification of $\widetilde X$ such that the boundary
\[
\partial X^+ = X^+ \smallsetminus \widetilde X
\]
and the union $\partial X^+ \cup E_s$ are divisors with simple normal crossings.   Then $\partial X^+ \cup E_i$ is a divisor with simple normal crossings for each $i$, because it is a union of components of $\partial X^+ \cup E_s$.

\begin{theorem} \label{pairs}
Fix an integer $r$, with $0 \leq r \leq s$.  Then the simple homotopy type of the filtered complex
\[
\Delta(E_1) \subset \cdots \subset \Delta(E_r) \subset \Delta(E_{r+1} \cup \partial X^+) \subset \cdots \subset \Delta(E_s \cup \partial X^+)
\]
is independent of the choices of compactification and resolution.
\end{theorem}

\noindent  Each of the complexes appearing in the theorem is a boundary complex.  For instance, $\Delta(E_i)$ is a boundary complex for $X^+ \smallsetminus \overline E_i$.  The proof of the theorem uses toroidal weak factorization of birational maps \cite{AKMW02, Wlodarczyk03} and extends arguments of Stepanov from \cite{Stepanov06}.  Similar methods were used earlier by Shokurov in his study of dual complexes of log canonical centers \cite{Shokurov00}.  

Theorem~\ref{only boundary} is the special case of Theorem~\ref{pairs} where $r = 0$, $s = 1$, and $Y_1$ is empty.  In the special case where $r = s = 1$, we recover the following simple homotopy invariance result for resolution complexes.

\begin{corollary}
Let $\pi$ be a weak log resolution of $X$ with respect to a closed subset $Y$.  Then the simple homotopy type of $\Delta( \pi^{-1}(Y))$ is independent of all choices.
\end{corollary}

\noindent  This corollary is due to Thuillier, for ordinary homotopy type under the additional assumption that $\pi$ is an isomorphism over $X \smallsetminus Y$ \cite{Thuillier07}, who generalized an earlier result of Stepanov for isolated singularities \cite{Stepanov06}.  One substantive improvement here is that $X$ need not be smooth away from $Y$ and resolutions with arbitrary discriminant are allowed.  This additional flexibility in choosing the resolution is helpful for examples and applications, such as the computation  of resolution complexes for singularities that are generic with respect to Newton polyhedra.  See Theorem~\ref{thm:hypersurface}.

In Section~\ref{affine section}, we study boundary complexes of affine varieties; some of the examples here were the original motivation for this project, through the relation between boundary complexes and tropicalizations for subvarieties of tori \cite{Tevelev07, Hacking08}.  In the case where $X$ is affine then, by theorems of Andreotti and Frankel in the smooth case \cite{AndreottiFrankel59} and Kar\v{c}jauskas in general \cite{Karcjauskas77, Karcjauskas79}, $X$ has the homotopy type of a regular CW complex of dimension $n$.  Since $\Gr_{2n}^W H^k(X)$ vanishes for $k$ less than $n$, the isomorphism (\ref{boundary homology display}) implies that the boundary complex of an affine variety has the rational homology of a wedge sum of spheres of dimension $n-1$.  For some special classes of affine varieties including surfaces, complements of hyperplane arrangements, and general complete intersections of ample hypersurfaces in the dense torus of a projective toric variety, and products of the above, we show that the boundary complex even has the homotopy type of a wedge sum of spheres; see Section~\ref{affine section}.  However, other homotopy types are also possible.  See Example~\ref{RP boundary} for an affine variety whose boundary complex has the homotopy type of a real projective space.

In Section~\ref{resolution section} we apply similar methods to study resolution complexes.  Consider now, for simplicity, the special case where $X$ is smooth away from an isolated singular point $x$.  Let $\pi: \widetilde X \rightarrow X$ be a weak log resolution with respect to $x$, with $E = \pi^{-1}(x)$.  In Section~\ref{resolution section}, we give a natural isomorphism identifying the reduced cohomology of the resolution complex $\Delta(E)$ with the weight zero part of the cohomology of $X$
\[
\widetilde H^{k-1}(\Delta(E); \Q) \cong W_0 \widetilde H^{k} (X),
\]
for $k \geq 1$.  If the singularity at $x$ is Cohen-Macaulay then each connected component of $\Delta(E)$ has the rational homology of a wedge sum of spheres of dimension $n - 1$, and if the singularity is rational then the resolution complex is connected and its top homology also vanishes, so the resolution complex has the rational homology of a point.  These facts follow directly from well-known results related to the theory of Du Bois singularities \cite[Proposition~3.2]{Ishii85} and \cite[Lemma~3.3]{Kovacs99}, and are highlighted in recent work of Arapura, Bakhtary, and W{\l}odarczyk, who also proved a generalization for non isolated rational singularities \cite{ABW09}.  

For special classes of singularities, including rational surface singularities, toric singularities \cite[Theorem~2]{Stepanov06}, and isolated rational hypersurface singularities of dimension three \cite[Corollary~3.3]{Stepanov08}, the resolution complex has the homotopy type of a wedge sum of spheres.  In Theorem~\ref{thm:hypersurface} we show that this is also the case for normal hypersurface singularities that are general with respect to their Newton polyhedra.  By work of Takayama on fundamental groups of exceptional divisors, resolution complexes of log terminal singularities are simply connected \cite{Takayama03}, but it is not known whether they are contractible.  Here we give an example of an isolated rational singularity whose resolution complex has the homotopy type of $\R\P^2$, obtained by partially smoothing a cone over a degenerate Enriques surface.  See Example~\ref{Enriques}.  This gives a negative answer to Stepanov's question on whether resolution complexes of rational singularities are contractible.  

It should be interesting to further investigate which simple homotopy types can appear as resolution complexes of rational and isolated Cohen-Macaulay singularities, and as boundary complexes of affine varieties, and what additional conditions are needed to guarantee that these complexes are contractible, simply connected, or have the homotopy type of a wedge sum of spheres.

\bigskip

\begin{remark}
Much of this paper was inspired by recent developments in tropical and nonarchimedean analytic geometry, especially the work of Hacking \cite{Hacking08}, Helm and Katz \cite{HelmKatz08}, and Thuillier \cite{Thuillier07}.  The connections between tropical geometry and Hodge theory have now been further developed by Katz and Stapledon in \cite{KatzStapledon10}.  The nonarchimedean methods are natural and powerful, but beyond the scope of this paper.  Our main results, and their proofs, can be presented without any tropical or nonarchimedean analytic language, and that is the style adopted here.

Since this paper appeared, Arapura, Bakhtary, and W{\l}odarczyk have used Stepanov's Lemma and weak factorization to prove further homotopy invariance results, with relations to weight filtrations and resolutions of singularities, in characteristic zero \cite{ABW11}.  We understand that these homotopy invariance results can also be proved over arbitrary perfect fields, through a natural extension of Thuillier's method.  Also, most recently, Kapovich and Koll\'{a}r have now classified the groups that occur as fundamental groups of resolution complexes of possibly non isolated rational singularities \cite[Theorem~42]{KapovichKollar11}.
\end{remark}

\subsection{Notation}  Throughout the paper, $X$ is an irreducible algebraic variety of dimension $n$ over the complex numbers.  All homology and cohomology groups of $X$ are taken with rational coefficients.  In Sections~\ref{affine section} and \ref{resolution section} we consider singular homology and cohomology of boundary complexes and resolution complexes with integer coefficients; this is clearly indicated where it occurs.

\subsection{Acknowledgments}  This project has its origins in discussions following a November 2005 talk by P.~Hacking on the homology of tropical varieties in the Tropical Geometry Seminar at the University of Michigan, organized by S.~Fomin and D.~Speyer.  It came to fruition during the special semester on algebraic geometry at MSRI in Spring 2009, where most of this paper was written and its main results presented \cite{boundarycomplexes-talk}.  This work has benefited from many conversations with friends and colleagues regarding the topology of algebraic varieties and their tropicalizations in the intervening years.  I thank, in particular,  D.~Abramovich, V.~Alexeev,~A.~Bj\"orner, T.~de~Fernex, W.~Fulton, S.~Galatius, P.~Griffiths, C.~Haase, P.~Hacking, R.~Hain, D.~Helm, P.~Hersh, M.~Kahle, M.~Kapranov, E.~Katz, J.~Koll{\'a}r, J.~Lewis,  C.~McCrory, M. Musta\c{t}\v{a}, B.~Nill, D.~Speyer, A.~Stapledon, B.~Sturmfels, B.~Totaro, Z.~Treisman, K.~Tucker, R.~Vakil, and J.~Yu.  I am also grateful to the referee for helpful comments and corrections.

\section{Dual complexes of divisors with simple normal crossings} \label{dual complexes}

Let $X$ be an algebraic variety, and let $D$ be a divisor in $X$ with simple normal crossings.  The dual complex $\Delta(D)$ is a topological space constructed by gluing simplices along faces, and is a $\Delta$-complex in the sense of \cite[Section~2.1]{Hatcher02}.  It has $r$ distinct vertices $v_1, \ldots, v_r$ corresponding to the irreducible components of $D$.  For each nonempty subset $I \subset \{1, \ldots, r \}$, we write $D_I$ for the intersection of the corresponding collection of components
\[
D_I = \bigcap_{i \in I} D_i,
\]
which is a smooth closed algebraic subset of pure codimension $\# I$ in $X$.  For each irreducible component $Y$ of $D_I$, the dual complex $\Delta(D)$ contains an embedded simplex $\sigma_Y$ whose vertices are exactly the $v_i$ for $i \in I$.   So $\Delta(D)$ has one vertex $v_i$ corresponding to each irreducible boundary divisor $D_i$, an edge joining $v_i$ and $v_j$ for each irreducible component of $D_i \cap D_j$, a two-dimensional face with vertices $v_i$, $v_j$, and $v_k$ for each irreducible component of $D_i \cap D_j \cap D_k$, and so on.  The relative interiors of these simplices are disjoint, and $\Delta(D)$ is the quotient space
\[
\Delta(D) = \coprod \sigma_V / \sim
\]
obtained by identifying $\sigma_{Y}$ with the face of $\sigma_Z$ spanned by its vertices, whenever $Y$ is contained in $Z$.  Since the vertices of each simplex are distinct, $\Delta(D)$ is a simplicial complex if and only if every collection of vertices spans at most one face, which is the case if and only if each intersection $D_I$ is either empty or irreducible.

\begin{remark}
We can sequentially blow up the irreducible components of all of the $D_I$, from smallest to largest, to obtain a projective birational morphism $\pi:X' \rightarrow X$ that is an isomorphism over $X \smallsetminus D$ such that $D' = \pi^{-1}(D)$ is a divisor with simple normal crossings and the intersection of any collection of components of $D'$ is either empty or irreducible.  Using this construction, we could choose resolutions so that all of the dual complexes we consider are simplicial, but the resulting complexes would have many more faces; the complex $\Delta(D')$ is the barycentric subdivision of $\Delta(D)$.  For computing examples, the option of working with $\Delta$-complexes with fewer faces is often preferred.
\end{remark}

\begin{remark} 
The dual complexes of divisors with simple normal crossings, although not simplicial complexes, are regular CW complexes, which means that the attaching maps are homeomorphisms.   It follows that these complexes are homeomorphic to the geometric realizations of their face posets \cite[\S 10.3.5]{Kozlov08}.  It is sometimes convenient to work directly with these face posets; for instance, one can apply the techniques of discrete Morse theory.
\end{remark}

\begin{example} \label{three lines}
Suppose $T$ is the complement of the three coordinate lines in the projective plane $\P^2$, which is a two-dimensional algebraic torus.  The boundary divisor $\P^2 \smallsetminus T$ is the union of the three coordinate lines, which has simple normal crossings.  The boundary complex $\Delta(\P^2 \smallsetminus T)$ has three vertices, one for each coordinate line, and since each pair of coordinate lines intersects in a single point, each pair of vertices is joined by a single edge, forming the boundary of a triangle.

Of course, there are many other compactifications of $T$ with simple normal crossing boundaries.  One way to construct other compactifications is by blowing up points in $\P^2 \smallsetminus T$.  For instance, if we blow up the point $[0:0:1]$, which is the intersection of the first two coordinate lines, then the strict transforms of those lines no longer meet, but both meet the exceptional divisor $E$.  On the other hand, if we blow up the point $[1:1:0]$ which is contained in only one of the coordinate lines, then the strict transforms of the coordinate lines still meet pairwise, and the new exceptional divisor $F$ meets only the third coordinate line.  These three different boundary complexes for $T$ are shown below.

\begin{picture}(400,110)(8,0)
\put(50,0){\makebox(0,0)[b]{$\Delta(\P^2 \smallsetminus T)$}}
\put(200,0){\makebox(0,0)[b]{$\Delta(\Bl_{[0:0:1]} \P^2 \smallsetminus T)$}}
\put(350,0){\makebox(0,0)[b]{$\Delta(\Bl_{[1:1:0]} \P^2 \smallsetminus T)$}}

\put(50,30){\circle*{3}}
\put(80,80){\circle*{3}}
\put(20,80){\circle*{3}}

\put(50,30){\line(3,5){30}}
\put(50,30){\line(-3,5){30}}
\put(20,80){\line(1,0){60}}

\put(200,30){\circle*{3}}
\put(230,80){\circle*{3}}
\put(170,80){\circle*{3}}
\put(185,55){\circle*{3}}

\put(175, 50){\makebox(0,0){$v_E$}}

\put(200,30){\line(3,5){30}}
\put(200,30){\line(-3,5){30}}
\put(170,80){\line(1,0){60}}

\put(350,30){\circle*{3}}
\put(380,80){\circle*{3}}
\put(320,80){\circle*{3}}
\put(420,104){\circle*{3}}

\put(428, 94){\makebox(0,0){$v_F$}}

\put(350,30){\line(3,5){30}}
\put(350,30){\line(-3,5){30}}
\put(320,80){\line(1,0){60}}
\put(380,80){\line(5,3){40}}
\end{picture}

Now, let $t$ be a point inside the torus $T$, and let $U$ be the quasiaffine variety $T \smallsetminus t$.  Then $\Bl_t \P^2$ is a smooth compactification of $U$ whose boundary is a divisor with simple normal crossings.  In this case, the boundary complex $\Delta\big((\Bl_t \P^2) \smallsetminus U\big)$ consists of a copy of $\Delta(\P^2 \smallsetminus T)$, corresponding to the three coordinate lines and their intersections, plus a single disjoint vertex corresponding to the exceptional fiber of the blowup.  In particular, boundary complexes are neither connected nor pure dimensional in general.
\end{example}

\begin{example} \label{two conics}
Suppose $X$ is the complement of two conics $C_1$ and $C_2$ in $\P^2$ that meet transversely.  Then $\P^2$ is a smooth compactification, with boundary divisor $C_1 \cup C_2$, and the boundary complex has two vertices, corresponding to the two conics, and four edges joining them, corresponding to the four points of intersection.  In this case, the boundary complex $\Delta(C_1 \cup C_2)$ is not a simplicial complex, since two vertices are joined by multiple edges. 

However, we could take instead the blowup $\Bl_{C_1 \cap C_2} \P^2$ of $\P^2$ at the four points of intersection, which is another smooth compactification whose boundary is a divisor with simple normal crossings, and the boundary complex $\Delta(\Bl_{C_1 \cap C_2} \P^2 \smallsetminus X)$ is simplicial.

\begin{picture}(420,110)(0,0)

\put(100,0){\makebox(0,0)[b]{$\Delta(C_1 \cup C_2)$}}
\put(300,0){\makebox(0,0)[b]{$\Delta(\Bl_{C_1 \cap C_2} \P^2 \smallsetminus X)$}}

\put(40,60){\circle*{3}}
\put(160,60){\circle*{3}}
\put(240,60){\circle*{3}}
\put(360,60){\circle*{3}}

\put(300,70){\circle*{3}}
\put(300,90){\circle*{3}}
\put(300,50){\circle*{3}}
\put(300,30){\circle*{3}}

\qbezier(40,60)(100,120)(160,60)
\qbezier(40,60)(100,80)(160,60)
\qbezier(40,60)(100,40)(160,60)
\qbezier(40,60)(100,0)(160,60)

\qbezier(240,60)(300,120)(360,60)
\qbezier(240,60)(300,80)(360,60)
\qbezier(240,60)(300,40)(360,60)
\qbezier(240,60)(300,0)(360,60)

\end{picture}

\end{example}

\begin{example} \label{RP boundary}
Let $T \cong (\C^*)^n$ be an algebraic torus.  Then the boundary divisor of any smooth toric compactification has simple normal crossings, and the boundary complex is naturally identified with the link of the vertex in the corresponding fan, which is a triangulation of a sphere of dimension $n-1$.

One smooth toric compactification of $T$ is the product of $n$ projective lines $(\P^1)^n$.  The involution $t \mapsto t^{-1}$ on $T$ extends to an involution on $(\P^1)^n$, and induces the antipodal map on the sphere $\Delta((\P^1)^n \smallsetminus T)$.  Then the quotient $X$ of $T$ by this involution is affine, and the quotient of $(\P^1)^n$ by the extended involution is a compactification of $X$ whose boundary is a divisor with simple normal crossings.  It follows that the boundary complex of $X$ has the homotopy type of the real projective space $\R\P^{n-1}$.  This example, suggested by J. Koll\'ar, shows that the boundary complex of an affine variety does not in general have the homotopy type of a wedge sum of spheres.
\end{example}

The following example shows that the simple homotopy type of every finite simplicial complex is realized by a boundary complex.

\begin{example}  \label{ex:arbitrarycx}
Let $\Delta$ be a subcomplex of the boundary complex of the $n$-simplex, which we may think of as a set of proper subsets of $\{0, \ldots, n \}$.  Fix homogeneous coordinates $[x_0: \cdots :x_n ]$ on $\P^n$ and, for each proper subset $I \subset \{ 0, \ldots, n \}$, let $L_I$ be the linear subspace where $x_i$ is nonzero if and only if $i$ is in $I$.  So $L_I$ has dimension $\# I - 1$.  Let
\[
X = \P^n \smallsetminus \Big( \bigcup_{I \in \Delta} L_I \Big).
\]
We claim that the barycentric subdivison of $\Delta$ is a boundary complex for $X$.  To see this, note that $\P^n$ is a smooth compactification of $X$.  Then, the space $X'$ obtained by blowing up first the zero dimensional spaces $L_I$ for $I \in \Delta$, then the strict transforms of the lines, and so on, is a smooth compactification of $X$ whose boundary $\partial X'$ is a divisor with simple normal crossings.  The components of $\partial X'$ correspond to the sets $I \in \Delta$, any nonempty intersection of components is irreducible, and an intersection is nonempty if and only if the corresponding subsets of $\{0, \ldots, n\}$ form a totally ordered chain.  Therefore, $\Delta(\partial X')$ is the barycentric subdivision of $\Delta$.
\end{example}

Similarly, the simple homotopy type of any finite filtered simplicial complex can be realized as one of the filtered complexes $\Delta(E_1) \subset \cdots \Delta(E_s)$ or $\Delta(\partial X^+) \subset \Delta(Y_2 \cup \partial X^+) \subset \cdots \subset \Delta(Y_r \cup \partial X^+)$ appearing in Theorem~\ref{pairs}.  In particular, for pairs $(\Delta(E \cup \partial X^+), \Delta(\partial X^+))$ such that the inclusion of the boundary complex is a homotopy equivalence, arbitrary Whitehead torsion is possible.

\section{Basics of weight filtrations} \label{weight basics}

Let $X$ be an algebraic variety of dimension $n$ over the complex numbers.  All homology and cohomology groups that we consider in this section are with rational coefficients.  

By work of Deligne \cite{Deligne75}, the rational cohomology groups of $X$ carry a canonical filtration
\[
W_0H^k(X) \subset \cdots \subset W_{2k}H^k(X) = H^k(X;\Q)
\]
that is strictly compatible with cup products, pullbacks under arbitrary morphisms, and long exact sequences of pairs.  This filtration is part of a mixed Hodge structure; the $j$th graded piece
\[
\Gr_j^W H^k(X) = W_j H^k(X)/W_{j-1}H^k(X)
\]
carries a natural pure Hodge structure of weight $j$.  See \cite{Durfee83} for a gentle and illuminating introduction to mixed Hodge theory, \cite{KulikovKurchanov98} for a survey of basic results, and \cite{PetersSteenbrink08} for a more comprehensive treatment with complete proofs and further references.  The proofs of many of the following basic properties of weight filtrations use Hodge theoretic arguments.  However, we will accept these properties as given and make no further reference to Hodge theory.  

\subsection{Combinatorial restrictions on weights} We say that $H^k(X)$ has weights in a subset $I \subset \{0, \ldots, 2k\}$ if the $j$th graded piece $\Gr_j^W H^k(X)$ vanishes for $j \not \in I$.  The weights that can appear in $H^k(X)$ depend on the compactness, smoothness, and dimension of $X$.  More precisely,
\begin{enumerate}
\item If $X$ is compact then $H^k(X)$ has weights in $\{0, \ldots, k\}$.
\item If $X$ is smooth then $H^k(X)$ has weights in $\{k, \ldots, 2k\}$.
\item If $k$ is greater than $n$ then $H^k(X)$ has weights in $\{0, \ldots, 2n\}$.
\end{enumerate}
In other words, if $X$ is compact then $W_k H^k(X)$ is $H^k(X)$, if $X$ is smooth then $W_{k-1}H^k(X)$ is zero, and if $n$ is less than $k$ then $W_{2n}H^k(X)$ is $H^k(X)$.  If $X$ is smooth and compact then (1) and (2) together say that $k$ is the only weight appearing in $H^k(X)$.

\begin{remark}
Although topological properties of $X$ such as compactness and dimension place combinatorial restrictions on the weights appearing in $H^k(X)$, the weight filtration is not a topological invariant of $X$; see \cite{SteenbrinkStevens84}.  Nevertheless, some pieces of the weight filtration are topological invariants, and pieces near the extreme ends of the weights allowed by these combinatorial restrictions tend to have particularly nice interpretations.  For instance, if $X$ is compact then $W_{k-1} H^k(X)$ is the kernel of the natural map from $H^k(X)$ to the intersection cohomology group $I\!H^k(X;\Q)$ \cite{Weber04, HanamuraSaito06}.
\end{remark}

 The singular homology, Borel-Moore homology, and compactly supported cohomology groups of $X$ carry weight filtrations with similar restrictions.  For simplicity, we consider only cohomology and compactly supported cohomology.  The weight filtration on compactly supported cohomology
\[
W_0H^k_c(X) \subset \cdots \subset W_{2k} H^k_c(X) = H^k_c(X;\Q)
\]
agrees with the weight filtration on $H^k(X)$ when $X$ is compact.

\subsection{Duality and exact sequences} When $X$ is smooth, the weight filtration on compactly supported cohomology satisfies the following duality with the weight filtration on cohomology.

\begin{PD}[\cite{PetersSteenbrink08}, Theorem~6.23]
If $X$ is smooth then the natural perfect pairing
\[
H^k_c(X) \times H^{2n-k}(X) \rightarrow \Q
\]
induces perfect pairings on graded pieces
\[
\Gr_j^W H^k_c(X) \times \Gr_{2n-j}^W H^{2n-k}(X) \rightarrow \Q,
\]
for $0 \leq j \leq 2k$.
\end{PD}

\noindent Similarly, long exact sequences of cohomology that arise naturally in geometry often descend to long exact sequences on graded pieces of the weight filtration.  The two sequences that we will use are the long exact sequence of a simple normal crossing pair, and the Mayer-Vietoris sequence of the mapping cylinder for a resolution of singularities.  We follow the usual convention that $W_j H^k(X)$ is $H^k(X;\Q)$ for $j$ greater than $2k$.

\begin{pairs} [\cite{PetersSteenbrink08}, Proposition~5.54]
Let $D$ be a divisor with simple normal crossings in $X$.  Then the long exact sequence of the pair $(X,D)$
\[
\cdots \rightarrow H^k(X) \rightarrow H^k(D) \rightarrow H^{k+1}_c(X \smallsetminus D) \rightarrow H^{k+1} (X) \rightarrow \cdots
\]
induces long exact sequences of graded pieces
\[
\cdots \rightarrow \Gr_j^W H^k(X) \rightarrow \Gr_j^W H^k(D) \rightarrow \Gr_j^W H^{k+1}_c(X \smallsetminus D) \rightarrow \Gr_j^W H^{k+1}(X) \rightarrow \cdots
\]
for $0 \leq j \leq 2n$.
\end{pairs}

\begin{MV}  [\cite{PetersSteenbrink08}, Corollary~5.37]
Let $\pi: X' \rightarrow X$ be a proper birational morphism with discriminant $V$ and exceptional locus $E$.  Then the long exact sequence of the mapping cylinder
\[
\cdots \rightarrow H^k(X) \rightarrow H^k(X') \oplus H^k(V) \rightarrow H^k(E) \rightarrow H^{k+1}(X) \rightarrow \cdots
\]
induces long exact sequences of graded pieces
\[
\cdots \rightarrow \Gr_j^W H^k(X) \rightarrow \Gr_j^W H^k(X') \oplus \Gr_j^W H^k(V) \rightarrow \Gr_j^W H^k(E) \rightarrow \Gr_j^W H^{k+1}(X) \rightarrow \cdots
\]
for $0 \leq j \leq 2n$.
\end{MV}

Using the basic tools above, we can express the cohomology of $X$ in terms of the cohomology groups of those of a compactification of a resolution, the boundary and exceptional divisors, and the discriminant.  When the discriminant is reasonably well-understood, this reduces many questions about the weight filtration on the cohomology of $X$ to questions about the weight filtrations on the cohomology of the boundary and exceptional divisors, which can be understood combinatorially as follows.

\subsection{Divisors with simple normal crossings}  Let $D$ be a divisor with simple normal crossings in $X$, and assume that $D$ is complete.  Let $D_1, \ldots, D_r$ be the irreducible components of $D$.  Then there is a combinatorial complex of $\Q$-vector spaces
\[
0 \rightarrow \bigoplus_{i=1}^r H^j(D_i) \xrightarrow{d_0} \bigoplus_{i_0 < i_1} H^j(D_{i_0} \cap D_{i_1}) \xrightarrow{d_1} \bigoplus_{i_0 < i_1 < i_2} H^j(D_{i_0} \cap D_{i_1} \cap D_{i_2}) \xrightarrow{d_2} \cdots,
\]
with differentials given by signed sums of restriction maps, and the cohomology of this complex gives the $j$th graded pieces of the weight filtrations on the cohomology groups of $D$ \cite[Chapter~4, \S2]{KulikovKurchanov98}.  More precisely, there are natural isomorphisms 
\[
\Gr_j^W H^{i+j}(D) \cong \frac{\ker d_i}{\im d_{i-1}},
\]
for all $i$.  In the special case when $j$ is zero, the complex above computes the cellular cohomology of the dual complex $\Delta(D)$, so we obtain natural isomorphisms
\[
W_0 H^j(D) \cong H^j (\Delta(D)),
\]
for all $j$.  

In many cases, these tools allow one to express the graded pieces of the weight filtration on the cohomology of $X$ combinatorially in terms of natural maps between cohomology groups of smooth strata in a suitable resolution of a compactification, together with the cohomology groups of the discriminant.  See Sections~\ref{boundary section} and \ref{singularities section} for details.  When the weight filtration on the cohomology groups of the discriminant are particularly simple, as is the case when $X$ has isolated singularities or, more generally, when the discriminant locus of some resolution is smooth and compact, then the resulting expressions are particularly satisfying.  See Theorem~\ref{smooth compact discriminant} and Remark~\ref{isolated duality}.

\begin{example}
We illustrate these basic results by describing the weight filtration on the first cohomology group of a singular punctured curve.  Let $X$ be a curve of geometric genus one, with one node and three punctures.  We begin by describing the first homology group.  Consider a loop passing through the node, two standard loops around the genus, and loops around two of the three punctures, as shown.

\begin{center}
\resizebox{3in}{1.8in}{\includegraphics{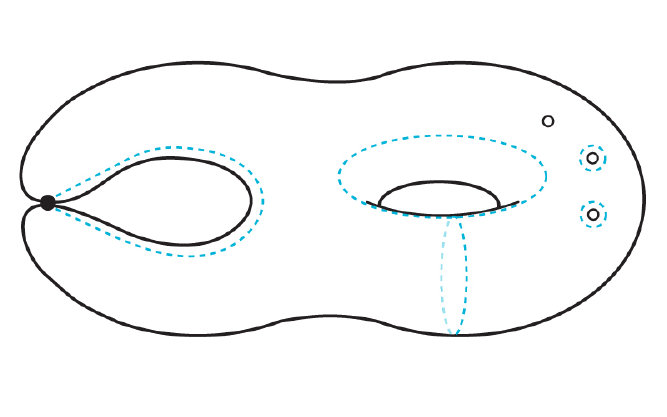}}
\end{center}

\noindent   These loops form a basis for $H_1(X)$.  In the dual basis for cohomology, the loop through the node corresponds to a generator for $W_0 H^1(X)$, the loops around the genus correspond to generators for $\Gr_1^W H^1(X)$, and the loops around the punctures correspond to generators for $\Gr_2^W H^1(X)$.  These correspondences can be can be deduced from the exact sequences discussed above.  For instance, the Mayer-Vietoris sequence for the normalization map $\widetilde X \rightarrow X$, whose exceptional divisor $E$ consists of two points in the preimage of the node, gives an exact sequence
\[
0 \rightarrow \widetilde H^0(E) \rightarrow H^1(X) \rightarrow H^1(\widetilde X) \rightarrow 0.
\]
Since $\widetilde H^0(E)$ and $H^1(\widetilde X)$ have pure weights zero and one, respectively, it follows that $W_0 H^1(X)$ is the kernel of the pullback map to $H^1(\widetilde X)$.  In particular, $W_0 H^1(X)$ is the subspace orthogonal to the image of $H_1(\widetilde X)$, which is spanned by the loops around the genus and the loops around the punctures.  This shows that $W_0 H^1(X)$ has rank one, and is spanned by the dual basis vector corresponding to the loop through the node.  The fact that the basis vectors dual to the loops around the genus span $\Gr_1^WH^1(X)$ can be seen similarly, using Poincar\'e Duality and the exact sequence for the pair $(X^+, D)$, where $X^+$ is the smooth compactification of $\widetilde X$ and $D$ is the boundary divisor, consisting of three points filling the punctures.

\smallskip 
\begin{center}
\resizebox{3in}{1.5in}{\includegraphics{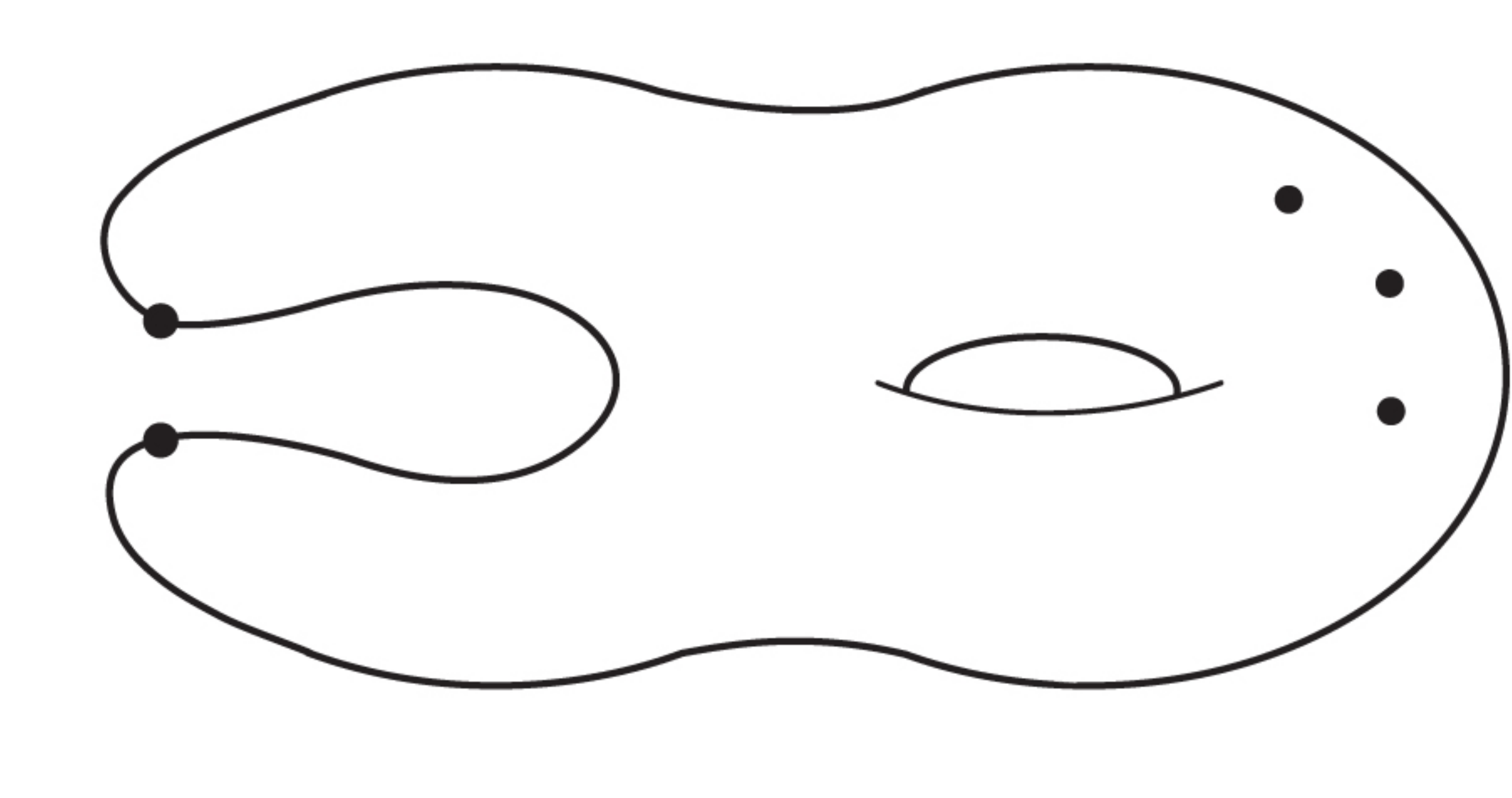}}
\end{center}

\noindent Altogether, these elementary arguments give natural isomorphisms
\[
\Gr_j^W H^1(X) \cong \left \{ \begin{array}{ll} \widetilde H^0(E) & \mbox{ for } j = 0, \\
										        H^1(X^+) & \mbox{ for } j = 1, \\
										        \widetilde H_0(D) & \mbox{ for } j = 2,
										        \end{array} \right.
\]
which are special cases of the complete characterization of the graded pieces of the weight filtration for varieties whose singular locus is smooth and proper, given in Theorem~\ref{smooth compact discriminant}.\footnote{For weight one, Theorem~\ref{smooth compact discriminant}  identifies $\Gr_1^W H^1(X)$ with $W_1 H^1(\widetilde X)$.  The additional isomorphism $W_1 H^1(\widetilde X) \cong H^1 (X^+)$ follows from Poincar\'e Duality and the pair sequence for $(X^+, D)$, using the fact that $X$ is a curve.}
\end{example}

\section{Weights from the boundary} \label{boundary section}

As in the previous two sections, $X$ is a variety of dimension $n$ over the complex numbers and all cohomology groups are with rational coefficients.

Here we use the basic tools from the previous section to express the $j$th graded piece of the weight filtration for large $j$ combinatorially in terms of cohomology groups of smooth complete boundary strata on a suitable compactification of a resolution.  Our approach, using standard exact sequences from topology plus knowledge of the cohomology of complete divisors with simple normal crossings, is similar in spirit to that followed by El Zein \cite{ElZein83}, but substantially less technical, since we are concerned only with the weight filtration and not the full mixed Hodge structure.  

\begin{remark}
The observation that the weight filtration can be understood separately from the mixed Hodge structure is not new.  For instance, Totaro has used the work of Guill\'en and Navarro Aznar \cite{GuillenNavarroAznar02} to introduce a weight filtration on the cohomology of real algebraic varieties \cite{Totaro02}, which is the subject of recent work of McCrory and Parusi\'nski \cite{McCroryParusinski09}.
\end{remark}

We begin by considering the case where $X$ is smooth, expressing pieces of the weight filtration on $H^*_c(X)$ in terms of the cohomology groups of a suitable smooth compactification and its boundary divisor.  The following result is known to experts but, lacking a suitable reference, we include a proof.

\begin{proposition}  \label{compactification weights}
Let $X$ be a smooth variety, with $X^+$ a smooth compactification such that the boundary $\partial X^+$ is a divisor with simple normal crossings. Then the associated graded pieces of the weight filtration $\Gr_j^W H^k_c(X)$ vanish for negative $j$ and for $j > k$, and there are natural isomorphisms
\[
W_j H^k_c(X) \cong W_j H^{k-1}(\partial X^+) \mbox{ for } 0 \leq j \leq k-2.
\]
The remaining two pieces of the weight filtration on $H^k_c(X)$ are given by
\[
W_{k-1} H^k_c(X) \cong \coker \left[ H^{k-1}(X^+) \rightarrow H^{k-1}(\partial X^+) \right],
\]
and
\[
\Gr_k^W H^k_c(X) \cong \ker \left[ H^{k}(X^+) \rightarrow H^k(\partial X^+) \right].
\]
\end{proposition}

\begin{proof}
The proposition follows directly from the long exact pair sequence
\[
\cdots \rightarrow H^{k-1}(X^+) \rightarrow H^{k-1}(\partial X^+) \rightarrow H^k_c(X) \rightarrow H^k(X^+) \rightarrow H^k(\partial X^+) \rightarrow \cdots,
\]
since $H^k(X^+)$ has weight $k$ and $H^k(\partial X^+)$ has weights between zero and $k$.  
\end{proof}

Next, we consider how the weight filtration on the ordinary cohomology of a singular variety relates to that of a resolution.

\begin{proposition}  \label{proper weights}
Let $X$ be a variety of dimension $n$ over the complex numbers, and let $\pi: \widetilde X \rightarrow X$ be a proper birational morphism from a smooth variety.  Suppose that
\begin{enumerate}
\item The weight $j$ is greater than $2n -2$, or
\item The discriminant of $\pi$ is proper, and the weight $j$ is greater than $k$.
\end{enumerate}
Then there is a natural isomorphism $\Gr_j^W H^k(X) \cong \Gr_j^W H^k(\widetilde X)$.
\end{proposition}

\begin{proof}
Let $V$ be the discriminant of $\pi$, and let $E = \pi^{-1}(V)$ be the exceptional divisor.  The Mayer-Vietoris sequence of mixed Hodge structures associated to $\pi$ has the form
\[
\cdots \rightarrow H^{k-1}(E) \rightarrow H^k(X) \rightarrow H^k(\widetilde X) \oplus H^k(V) \rightarrow H^k(E) \rightarrow \cdots.
\]
The proposition follows, since $\Gr_j^W H^k(V)$ and $\Gr_j^W H^k(E)$ vanish for $j$ greater than $2n - 2$, and for $j > k$ if $V$ is proper.
\end{proof}

Using the preceding propositions and Poincar\'e duality, we now show that the homology of the boundary complex gives the $2n$th graded piece of the weight filtration.

\begin{theorem}  \label{boundary homology}
Let $\pi: \widetilde X \rightarrow X$ be a proper birational morphism from a smooth variety, and let $X^+$ be a compactification of $\widetilde X$ whose boundary $\partial X^+$ is a divisor with simple normal crossings.  Then, for all nonnegative integers $k$, there is a natural isomorphism
\[
\widetilde H_{k-1}\big( \Delta(\partial X^+)\big) \cong \Gr_{2n}^W H^{2n-k}(X).
\]
\end{theorem}

\begin{proof}
By Proposition~\ref{proper weights}, $\Gr_{2n}^W H^{2n-k}(X)$ is naturally isomorphic to $\Gr_{2n}^W H^{2n-k}(\widetilde X)$.  Since $\widetilde X$ is smooth, the latter is Poincar\'e dual to $W_0 H^k_c (\widetilde X)$.  Now, by Proposition~\ref{compactification weights}, there is a natural isomorphism
\[
 W_0 H^k_c(X) \cong W_0 \widetilde H^{k-1}(\partial X^+; \Q),
 \]
and since $\partial X^+$ is a divisor with simple normal crossings, the weight zero cohomology group $W_0 \widetilde H^{k-1}(\partial X^+)$ is isomorphic to the cohomology group of the dual complex $\widetilde H^{k-1}(\Delta(\partial X^+); \Q)$.  The theorem follows, since the cohomology group $\widetilde H^{k-1}(\Delta(\partial X^+); \Q)$ is dual to the homology group $\widetilde H_{k-1} (\Delta(\partial X^+); \Q)$.
\end{proof}

\section{Simple homotopy equivalences from weak factorizations} \label{homotopy equivalences}

Having proved that the reduced homology of the boundary complex of $X$ is naturally isomorphic to the $2n$th graded piece of the weight filtration on the singular cohomology of $X$, we now show that the homotopy type of the boundary complex is independent of all choices, using toroidal weak factorization of birational maps and a lemma of Stepanov. 

Roughly speaking, the toroidal weak factorization theorem of Abramovich, Karu, Matsuki, and W{\l}odarczyk says that any birational map that is an isomorphism away from divisors with simple normal crossings can be factored as a series of blowups and blowdowns along admissible centers, which are smooth subvarieties that have simple normal crossings with the given divisors.  See \cite{AKMW02, Wlodarczyk03} and \cite[Theorem~5-4-1]{Matsuki00} for details.  The preimages of the given divisor under such a blowup again has simple normal crossings, and Stepanov's Lemma says that the dual complex of the new divisor is either unchanged or is obtained from the old dual complex by a combinatorial operation that preserves homotopy type, such as barycentric subdivision or gluing on a cone over a contractible subcomplex, as in Examples~\ref{three lines} and \ref{two conics}, above.

Let $X$ be an algebraic variety and let $D$ be a divisor on $X$ with simple normal crossings.  Recall that a subvariety $Z \subset X$ has \emph{simple normal crossings} with $D$ if every point $z$ in $Z$ has a neighborhood with local coordinates in which each component of $D$ that contains $z$ is the vanishing locus of a single coordinate and $Z$ is the vanishing locus of some subset of the coordinates \cite[p.~137]{Kollar07}.

\begin{Stepanov} \cite{Stepanov06}
Let $Z \subset X$ be a closed subvariety that has simple normal crossings with $D$.  Let $\pi$ be the blowup along $Z$.  Then $D' = \pi^{-1}(D)$ is a divisor with simple normal crossings, and $\Delta(D')$ is simple homotopy equivalent to $\Delta(D)$.  More precisely,
\begin{enumerate}
\item If $Z$ is not contained in $D$, then $\Delta(D')$ and $\Delta(D)$ are naturally isomorphic.
\item If $Z$ is an irreducible component of $D_I$, then $\Delta(D')$ is naturally isomorphic to the barycentric subdivision of $\Delta(D)$ along the face corresponding to $Z$.
\item If $Z$ is properly contained in a component $V$ of the intersection $D_I$ of all components of $D$ that contain $Z$, then $\Delta(D')$  is obtained from $\Delta(D)$ by gluing on the cone over $\sigma_V$ and then attaching copies of cones over faces that contain $\sigma_V$.
\end{enumerate}
\end{Stepanov}

\noindent 

\begin{remark} 
In case (3), a strong deformation retraction from $\Delta(D')$ to $\Delta(D)$ can be given by flowing the cone point to a vertex of $\sigma_V$ and extending linearly.  Such discrete Morse flows preserve simple homotopy type \cite{Forman98}.  For details on the theory of simple homotopy types and related invariants, such as Whitehead torsion, see \cite{Cohen73}.
\end{remark}

Using Stepanov's Lemma, we can give simple homotopy equivalences between dual complexes of simple normal crossings divisors whenever the complexes can be obtained by blowups and blowdowns along smooth admissible centers.  Furthermore, we can often choose these homotopy equivalences to be compatible with suitable subcomplexes.  Here a \emph{subcomplex} of a $\Delta$-complex is a closed subset that is a union of simplices, and a \emph{filtered $\Delta$-complex} is a complex $\Delta$ with a chain of subcomplexes
\[
\Delta_1 \subset \cdots \subset \Delta_t = \Delta.
\]
Simple homotopy equivalence is an equivalence relation on regular CW complexes that is generated by barycentric subdivisions, elementary collapses, and their inverses.  A \emph{simple homotopy equivalence} between filtered $\Delta$-complexes $\big(\Delta, \{ \Delta_i \}\big)$ and $\big(\Delta', \{\Delta'_j\}\big)$ is a simple homotopy equivalence $\Delta \simeq \Delta'$ that respects the filtrations and induces simultaneous simple homotopy equivalences $\Delta_i \simeq \Delta'_i$ for all $i$, and two filtered complexes have the same \emph{simple homotopy type} if there exists a simple homotopy equivalence between them.

The complexes that we consider are dual complexes of simple normal crossings divisors in compactifications of weak log resolutions whose boundaries are divisors with simple normal crossings. 

\begin{definition}
Let $X$ be an algebraic variety, with $Y_1, \ldots,Y_s$ a collection of closed algebraic subsets.  A \emph{weak log resolution} of $X$ with respect to $Y_1, \ldots, Y_s$ is a proper birational morphism from a smooth variety to $X$ such that the preimage of each $Y_i$ is a divisor with simple normal crossings.  
\end{definition}

\noindent We do not require a log resolution to be an isomorphism over every smooth point in the complement of $Y$.  In particular, even when $Y$ is contained in the singular locus, we allow \emph{weak} resolutions of $X$ that are not necessarily isomorphisms over the smooth locus.

The main tools in the proof are Stepanov's Lemma and toroidal weak factorization of birational maps.  In order to apply toroidal weak factorization, we reduce to the case of comparing two log resolutions with the same discriminants in $X$.  The key step in this reduction is the following lemma which says that we can increase the discriminant without changing the homotopy types of the filtered complexes.   Recall that the \emph{discriminant} $\discr (\pi)$ of a projective birational morphism $\pi$ to $X$ is the smallest closed subset $V$ in $X$ such that $\pi$ is an isomorphism over $X \smallsetminus V$.

\begin{lemma}  \label{enlarge discriminant}
Let $X$ be an algebraic variety, with $Y_1 \subset \cdots \subset Y_s$ a chain of closed algebraic subset and $\pi_1 : X_1 \rightarrow X$ a weak log resolution with respect to $Y_1, \ldots, Y_s$.  Then, for any closed subset $V$ of $X$ that contains $\discr(\pi_1)$, there is a log resolution $\pi_2: X_2 \rightarrow X$ with respect to $Y_1, \ldots, Y_s$ such that
\begin{enumerate}
\item The log resolution $\pi_2$ is an isomorphism over $X \smallsetminus V$,
\item The preimage of $V$ in $X_2$ is a divisor with simple normal crossings, and
\item The filtered complex
\[
\Delta(\pi_2^{-1}(Y_1)) \subset \cdots \subset \Delta(\pi_2^{-1}(Y_s))
\]
is simple homotopy equivalent to $\Delta(\pi_1^{-1}(Y_1)) \subset \cdots \subset \Delta(\pi_2^{-1}(Y_s))$.
\end{enumerate}
\end{lemma}

\begin{proof}
By Hironaka's strong principalization theorem \cite[Theorem~3.26]{Kollar07}, there is a proper birational morphism $\pi_2$ constructed as the composition of $\pi_1$ with a series of blowups along smooth centers in the preimage of $V$ that have simple normal crossings with the preimage of $Y_s$ such that $\pi_2$ satisfies (1) and (2).  We now use Stepanov's Lemma to prove that (3) is satisfied.  It will suffice to consider how the filtered complex changes under the blowup of $X_1$ along a smooth subvariety $Z$ that has normal crossings with the preimage of $Y_s$.

Let $D(i) = \pi_1^{-1} (Y_i),$ and let $D'(i)$ be the preimage of $D(i)$ under the blowup along $Z$.  We must show that the filtered complex
\[
\Delta(D'(1)) \subset \cdots \subset \Delta(D'(s)) = \Delta'
\]
is homotopy equivalent to $\Delta(D(1)) \subset \cdots \subset \Delta(D(s)) = \Delta$.  If $D(s)$ does not contain $Z$ then, by part (1) of Stepanov's Lemma, the two filtered complexes are naturally isomorphic.  Otherwise, we consider two cases according to how $Z$ sits inside $D(s)$.  

First, suppose $Z$ is a component of some intersection $D_I$ of components of $D(s)$.  By Stepanov's Lemma, $\Delta'$ is the barycentric subdivision of $\Delta$ along the face $\sigma_Z$.  However, the induced homeomorphism $\Delta' \simeq \Delta$ does not necessarily map $\Delta'_i$ into $\Delta_i$ for all $i$; if $D(i)$ contains some but not all of the components of $D(s)$ that contain $Z$, then $\Delta'_i$ is obtained from $\Delta_i$ by gluing on the cone over the face $\sigma_V$ corresponding to the smallest stratum of $D(i)$ that contains $Z$ and then attaching cones over faces that contain $\sigma_V$.  In particular, $\Delta'_i$ may not be homeomorphic to $\Delta_i$.  Nevertheless, we give a compatible homotopy equivalences $\Delta'_i \simeq \Delta_i$ for all $i$, as follows.

Choose $i$ as small as possible so that $D(i)$ contains $Z$, and let $D_j$ be a component of $D(i)$ that contains $Z$.  Then there is a piecewise-linear homotopy equivalence from $\Delta'$ to $\Delta$ taking the vertex $v_E$ corresponding to the exceptional divisor to the vertex $v_j$ corresponding to $D_j$, and preserving the vertices corresponding to all other $D_k$.  The homotopy equivalence corresponds to the following map $p$ from the face poset of $\Delta'$ to the face poset of $\Delta$.  For any face $\sigma$ in $\Delta'$ that contains $E$, $p(\sigma)$ is the unique face of $\Delta$ spanned by the vertices of $\sigma$ other than $v_E$, together with $v_j$; if $\sigma$ does not contain $v_E$, then $p(\sigma) = \sigma$.  This homotopy equivalence is induced by the discrete Morse flow, in the sense of \cite{Forman98}, that pairs each face $\sigma$ that contains $v_E$ but not $v_j$ into the unique face spanned by $\sigma$ and $v_j$. 

\begin{picture}(420,150)(0,0)
\put(320,0){\makebox(0,0)[b]{ \scalebox{.7}{ \includegraphics{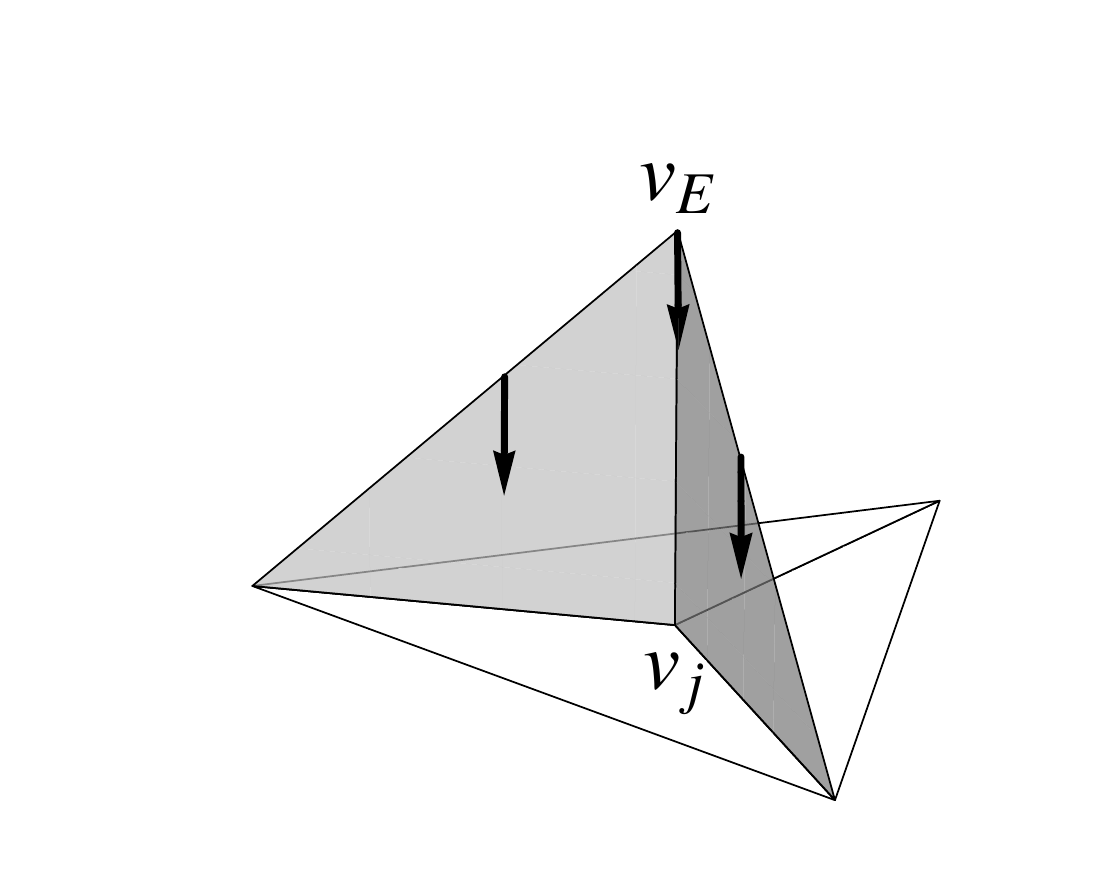}}}}
\put(90,-20){\makebox(0,0)[b]{ \scalebox{.5}{ \includegraphics{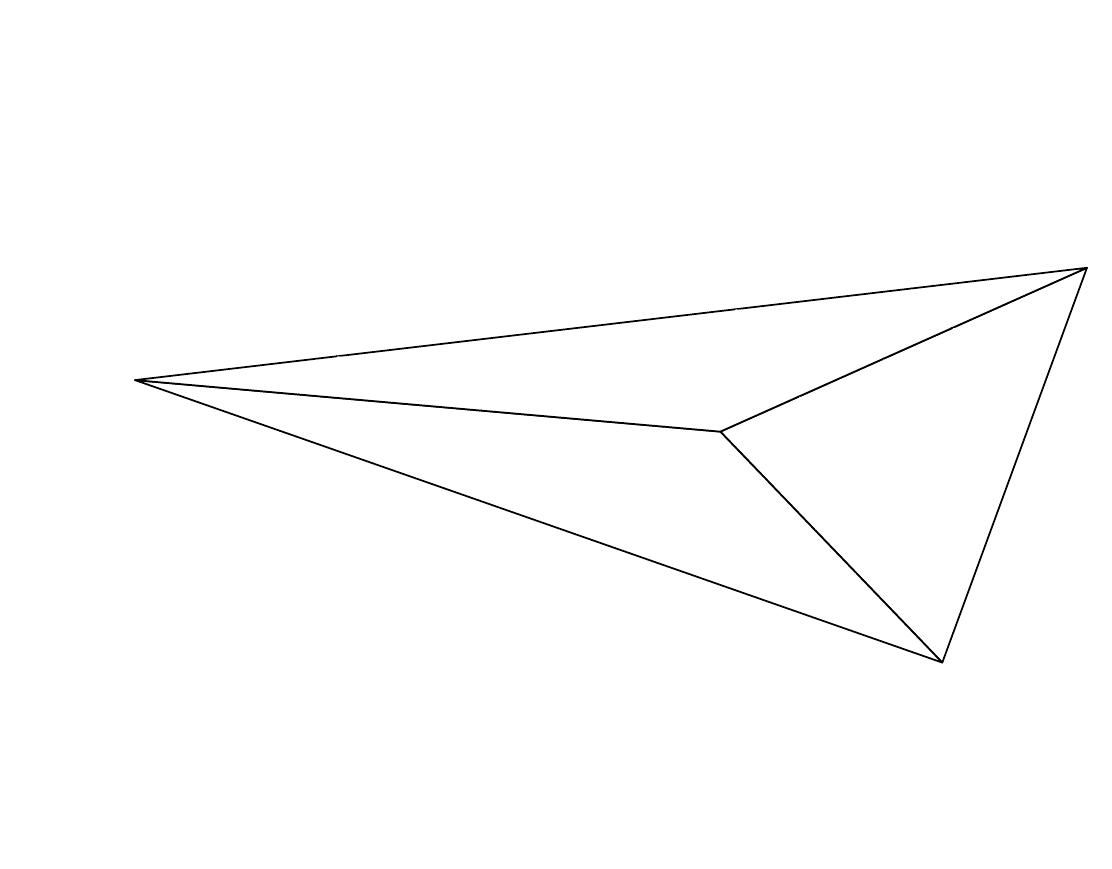}}}}
\put(90,0){\makebox(0,0){$\Delta$}}
\put(320,0){\makebox(0,0){$\Delta'$}}
\end{picture}

\bigskip

\noindent This discrete Morse flow gives compatible homotopy equivalences between $\Delta'_i$ and $\Delta_i$, for all $i$.

\smallskip

Otherwise, $Z$ is properly contained in a component $V$ of the intersection of all components of $D$ that contain it.  By part (3) of Stepanov's Lemma, $\Delta'$ is constructed from $\Delta$ by gluing on the cone over $\sigma_V$ and then attaching copies of cones over faces that contain $\sigma_V$.  Again, we choose $i$ as small as possible so that $D(i)$ contains $Z$, with $D_j$ a  component of $D(i)$ that contains $Z$, and there are compatible homotopy equivalences from $\Delta'_i$ to $\Delta_i$ taking $v_E$ to $v_j$, as required.
\end{proof}

\begin{proof}[Proof of Theorem~\ref{pairs}]
Suppose $X_1^+$ and $X_2^+$ are two such compactifications of weak log resolutions.  By Lemma~\ref{enlarge discriminant}, we may assume there is a closed algebraic subset $V \subset X$ such that $\pi_1: \widetilde X \rightarrow X$ and $\pi_2: \widetilde X_2 \rightarrow X$ are isomorphisms over $X \smallsetminus V$, and the preimages of $V$, as well as $\pi_1^{-1}(V) \cup \partial X_1^+$ and $\pi_2^{-1}(V) \cup \partial X_2^+$ are divisors with simple normal crossings.  We now compare the filtered complexes corresponding to these compactifications of resolutions using weak factorization.

By toroidal weak factorization of birational maps, there exist smooth compactifications $Z_0, \ldots, Z_t$ of $X \smallsetminus V$ in which the complement of $X \smallsetminus V$ is a divisor with simple normal crossings and a sequence of birational maps
\[
X_1^+ = Z_0 \stackrel{\phi_1}{\dashrightarrow} Z_1 \stackrel{\phi_2}{\dashrightarrow} \cdots  \stackrel{\phi_t}{\dashrightarrow} Z_t = X_2^+
\]
such that
\begin{enumerate}
\item Each $\phi_j$ is the identity on $X \smallsetminus V$;
\item Either $\phi_j$ or $\phi_j^{-1}$ is the blowup along a smooth variety that has simple normal crossings with the complement of $X \smallsetminus V$;
\item There is an index $j_0$ such that the birational map $Z_j \dashrightarrow X_1^+$ is a projective morphism for $j \leq j_0$ and $Z_j \dashrightarrow X_2^+$ is a projective morphism for $j \geq j_0$.
\end{enumerate}
In particular, each $Z_j$ contains an open subset $U_j$ that maps projectively onto $X$.

Now, let $E_i(j)$ be the preimage of $Y_i$ in $U_j$, and let $F_i(j)$ be $E_i(j) \cup (Z_j \smallsetminus U_j)$.  For each $j$ we apply Stepanov's Lemma to either $\phi_j$ or $\phi_{j}^{-1}$, whichever is a blowup.  As in the proof of Lemma~\ref{enlarge discriminant}, there is a discrete Morse flow connecting $\Delta(F_s(j-1))$ to $\Delta(F_s(j))$ that takes the vertex corresponding to the exceptional divisor to a carefully chosen vertex corresponding to an irreducible divisor that contains the center of the blowup, inducing a simple homotopy equivalence of filtered complexes between
\[
\Delta(E_1(j-1)) \subset \cdots \subset \Delta(E_r(j-1)) \subset \Delta(F_{r+1}(j-1)) \subset \cdots \subset \Delta(F_s(j-1))
\]
and
\[
\Delta(E_1(j)) \subset \cdots \subset \Delta(E_r(j)) \subset \Delta(F_{r+1}(j)) \subset \cdots \subset \Delta(F_s(j)).
\]
The theorem follows, by composing these simple homotopy equivalences.
\end{proof}

\section{Boundary complexes of affine varieties}  \label{affine section}

Suppose $X$ is affine, $\pi: \widetilde X \rightarrow X$ is a log resolution, and $X^+$ is a log compactification of $\widetilde X$ with respect to the exceptional divisor $E$ of $\pi$.  As noted in the introduction, it follows from Theorem~\ref{boundary homology} and theorems of Andreotti and Frankel, and of Kar{\v c}jauskas that the boundary complex $\Delta(\partial X^+)$ has the rational homology of a wedge sum of spheres of dimension $n-1$.  If $n$ is one, then the boundary complex is just a finite set of points, and if $n$ is two then the boundary complex is a connected graph, and hence has the homotopy type of a wedge sum of circles.  In this section, we consider the homotopy types of boundary complexes of affine varieties of dimension at least three. 

\begin{remark}
Suppose $n$ is at least three.  By theorems of Whitehead and Hurewicz, a regular CW complex has the homotopy type of a wedge sum of spheres of dimension $n-1$ if and only if it is simply connected and has the integral homology of a wedge sum of spheres of dimension $n-1$.  Therefore, $\Delta(\partial X^+)$ has the homotopy type of a wedge sum of $n-1$-dimensional spheres if and only if it is simply connected and its homology is torsion free.
\end{remark}

 The following lemma is helpful for studying boundary complexes of products.  For topological spaces $\Delta_1$ and $\Delta_2$, we write $\Delta_1 * \Delta_2$ for the join
\[
\Delta_1 * \Delta_2 = \big( \Delta_1 \times \Delta_2 \times [0,1] \big) / \sim,
\]
where $\sim$ is the equivalence relation generated by $(x, y, 0) \sim (x, y', 0)$ and $(x,y,1) \sim (x',y,1)$.  See \cite[p.~18]{Hatcher02} for further details on joins.

\begin{lemma}  \label{join}
Let $X_1$ and $X_2$ be smooth varieties, with log compactifications $X_1^+$ and $X_2^+$, respectively.  Then the product $X_1^+ \times X_2^+$ is a log compactification of $X_1 \times X_2$, and there is a natural homeomorphism
\[
\Delta\big( \partial (X_1^+ \times X_2^+)\big) \cong \Delta(\partial X_1^+) * \Delta(\partial X_2^+).
\]
\end{lemma}

\begin{proof}
The product $X_1^+ \times X_2^+$ is smooth, and its boundary is the divisor
\[
\partial(X_1^+ \times X_2^+) = \big(\partial X_1^+ \times X_2^+\big) \cup \big(X_1^+ \times \partial X_2^+\big),
\]
which has simple normal crossings, by a suitable choice of coordinates on each factor.  Say $D_{11}, \ldots, D_{r1}$ and $D_{12}, \ldots, D_{s2}$ are the irreducible components of $\partial X_1^+$ and $\partial X_2^+$, respectively.  For any subsets of indices $I \subset \{1, \ldots, r\}$ and $J \subset \{1, \ldots, s \}$, let $D_{IJ}$ be the corresponding intersection
\[
D_{IJ} = \bigg(\bigcap_{i \in I} D_{i1}\bigg) \cap \bigg(\bigcap_{j \in J} D_{j2}\bigg).
\]
The irreducible components of $D_{IJ}$ are exactly the subvarieties $Y \times Z$ where $Y$ and $Z$ are irreducible components of $D_I \times X_2^+$ and $X_1^+ \times D_J$, respectively.  The corresponding simplex $\sigma_{Y \times Z}$ is naturally identified with the join
\[
\sigma_{Y \times Z} \cong \sigma_Y * \sigma_Z.
\]
Since the join of two $\Delta$-complexes is the union of the joins of their respective simplices with the natural identifications, the lemma follows.
\end{proof}

\begin{proposition}
If $\Delta(\partial X_1^+)$ and $\Delta(\partial X_2^+)$ are homotopy equivalent to wedge sums of $m_1$ spheres of dimension $n_1$, and $m_2$ spheres of dimension $n_2$, respectively, then the boundary complex of the product $\Delta\big(\partial X_1^+ \times X_2^+) \big)$ is homotopy equivalent to a wedge sum of $m_1 \cdot m_2$ spheres of dimension $n_1 + n_2 + 1$.
\end{proposition}

\begin{proof}
By Lemma~\ref{join}, $\Delta\big(\partial (X_1^+ \times X_2^+) \big)$ is naturally homeomorphic to the join $\Delta(\partial X_1^+) * \Delta(\partial X_2^+)$.  Since joins commute with homotopy equivalences, it will suffice to give a homotopy equivalence
\[
\vee^{m_1} S^{n_1} * \vee^{m_2} S^{n_2} \simeq \vee^{(m_1 \cdot m_2)} S^{n_1 + n_2 + 1},
\]
where $S^n$ denotes the sphere of dimension $n$.  Now wedge sums commute with pushouts, so for any topological spaces $\Delta_1$, $\Delta_2$, and $\Delta_3$, the join $(\Delta_1 \vee \Delta_2) * \Delta_3$ is naturally homeomorphic to the union of $\Delta_1 * \Delta_3$ and $\Delta_2 * \Delta_3$, glued along the join of a point with $\Delta_3$.  Since the join of a point with a complex is a cone over that complex, and hence contractible, it follows that there is a homotopy equivalence
\[
(\Delta_1 \vee \Delta_2) * \Delta_3 \simeq (\Delta_1 * \Delta_3) \vee (\Delta_2 * \Delta_3).
\]
The result then follows by induction on $m_1$ and $m_2$, since the join of spheres $S^{n_1} * S^{n_2}$ is homeomorphic to the sphere $S^{n_1 + n_2 + 1}$.
\end{proof}

We have seen that boundary complexes of products of curves and surfaces are homotopy equivalent to wedge sums of spheres.  Now we prove the same for complements of hyperplane arrangements and general complete intersections in the dense torus of a projective toric variety.  We begin with two combinatorial lemmas that are helpful for identifying CW complexes that are homotopy equivalent to a wedge sum of spheres.

\begin{lemma} \label{n skeleton}
Let $\Delta$ be the $n$-skeleton of a contractible CW complex.  Then $\Delta$ is homotopy equivalent to a wedge sum of spheres of dimension $n$.
\end{lemma}

\begin{proof}
If $n$ is zero or one then the lemma is clear.  Assume $n$ is at least two, and say $\Delta'$ is a contractible CW complex with $n$-skeleton $\Delta$.  Then, since $\Delta$ and $\Delta'$ have the same $2$-skeleta they must have isomorphic fundamental groups, so $\Delta$ is simply connected.  Similarly, the integral homology $H_k(\Delta, \Z)$ vanishes for $k < n$, and $H_n(\Delta, \Z)$ is free, since $\Delta$ is $n$-dimensional.  Therefore, $\Delta$ is simply connected and has the integral homology of a wedge sum of spheres of dimension $n$, and the lemma follows by the Whitehead and Hurewicz theorems.
\end{proof}

If $\Delta$ is a CW complex and $\sigma$ is a maximal cell in $\Delta$, then we say that the $d$-fold \emph{puckering} of $\Delta$ along $\sigma$ is the CW complex obtained from $\Delta$ by attaching $d-1$ new cells of dimension $\dim \sigma$ along the attaching map for $\sigma$; i.e. if $\Delta$ is a simplicial complex then we glue on $d-1$ new copies of $\sigma$ by identifying all of their boundaries with the boundary of $\sigma$.  See \cite[Section~2]{tvbs-jag} for further details on puckering operations and their significance in toric geometry.

\begin{lemma} \label{puckering}
Let $\Delta$ be a regular CW complex that is homotopy equivalent to a wedge sum of spheres of dimension $n$, and let $\cP$ be the $d$-fold puckering of a maximal $n$-dimensional cell in $\Delta$.  Then $\cP$ is homotopy equivalent to a wedge sum of spheres of dimension $n$.
\end{lemma}

\begin{proof}
If $n$ is zero then the lemma is trivial.  Otherwise, the attaching map for each of the new $n$-cells is null homotopic, since it can be contracted to a point in $\sigma$.  Therefore, $\cP$ is homotopy equivalent to the wedge sum of $\Delta$ with $d-1$ copies of the $n$ sphere. 
\end{proof}

\bigskip

Let $Y$ be a projective toric variety, and let $V_1, \ldots, V_k$ be ample locally principal hypersurfaces in $Y$, where $V_i$ is cut out by a section $s_i$ of an ample line bundle $L_i$.  Then the system $\{ V_1, \ldots, V_k \}$ is \emph{nondegenerate with respect to Newton polytopes} in the sense of Khovanskii \cite{Khovanskii77} if, for every torus orbit $O_\sigma$ in $Y$, the restricted sections $\{s_i|_{O_\sigma}\}$ cut out a smooth subvariety of codimension $k$ in $O_\sigma$.

\begin{theorem}
Let $Y$ be a projective toric variety with dense torus $T$, and let $\{ V_1, \ldots, V_k \}$ be a system of ample hypersurfaces in $Y$ that is nondegenerate with respect to Newton polytopes.  Let $X$ be the intersection in the dense torus
\[
 X \ = \ (V_1 \cap \cdots \cap V_k) \cap T.
\]
Then the boundary complex of any log compactification of $X$ is homotopy equivalent to a wedge sum of spheres of dimension $\dim X-1$.
\end{theorem}

\begin{proof}
Let $n$ be the dimension of $X$.  The theorem is clear if $n$ is at most two, so assume $n \geq 3$.  By Theorem~\ref{only boundary}, the homotopy type of the boundary complex does not depend on the choice of log compactification.  Let $Y' \rightarrow Y$ be a toric resolution of singularities.  The closure $X^+$ of $X$ in $Y'$ is smooth and transverse to $Y' \smallsetminus T$, by \cite[Theorem~2]{Khovanskii77}, so it will suffice to show that the boundary complex $\Delta(\partial X^+)$ is homotopy equivalent to a wedge sum of spheres of dimension $\dim X - 1$;  We will prove this by comparing $\Delta(\partial X^+)$ to the fans corresponding to $Y$ and $Y'$.

Let $\Sigma$ and $\Sigma'$ be the fans corresponding to $Y$ and $Y'$, respectively.  Let $\Sigma(n)$ be the link of the vertex in the $n$-skeleton of $\Sigma$, which is homotopy equivalent to a wedge sum of $n-1$-dimensional spheres by Lemma~\ref{n skeleton}.  For each $n$-dimensional cone $\sigma \in \Sigma$, let $d(\sigma)$ be the intersection number
\[
d(\sigma) = (V_1 \cdots V_k \cdot V(\sigma)),
\]
which is a positive integer since the $V_i$ are ample and locally principal.  Let $\cP$ be the regular CW complex obtained by $d(\sigma)$-fold puckering of the maximal face corresponding to $\sigma$ in $\Sigma(n)$.  Then $\cP$ is homotopy equivalent to a wedge sum of spheres of dimension $n-1$, by Lemma~\ref{puckering}.  We claim that the boundary complex $\Delta(\partial X^+)$ is naturally homeomorphic to the subdivision
\[
\cP' \cong \cP \times _{\Sigma_{n}} \Sigma'_n
\]
of $\cP$ induced by the subdivision $\Sigma'$ of $\Sigma$, and hence also homeomorphic to a wedge sum of spheres of dimension $n-1$, as required.

Say $\tau$ is a cone in $\Sigma'$.  Then the intersection of $X^+$ with the corresponding $T$-invariant subvariety $V(\tau)$ is either 
\begin{enumerate}
\item A smooth irreducible subvariety of codimension $\dim \tau$ if $\tau$ is contained in the $n-1$-skeleton of $\Sigma$.
\item A disjoint union of $d(\sigma)$ smooth irreducible subvarieties of codimension $\dim \tau$, corresponding to the $d(\sigma)$ distinct intersection points of $X^+$ with $V(\sigma)$, if the relative interior of $\tau$ is contained in the relative interior of an $n$-dimensional cone $\sigma \in \Sigma$, or
\item Empty, if $\tau$ is not contained in the $n$-skeleton of $\Sigma$.
\end{enumerate}
The inclusions of these subvarieties are the natural ones, respecting the labellings by points in $X^+ \cap V(\sigma)$, and it follows that $\Delta(\partial X^+)$ is naturally isomorphic to the subdivision $\cP'$ of $\cP$ induced by $\Sigma'$, as required.
\end{proof}

\begin{remark}
In the proof above, we used the topological fact that the link $\Sigma(n)$ of the vertex in the $n$-skeleton of the fan corresponding to a projective toric variety is homotopy equivalent to a wedge sum of spheres, because it is the $(n-1)$-skeleton of a contractible CW complex (Lemma~\ref{n skeleton}).  There are other combinatorial proofs that this link is homotopy equivalent to a wedge sum of spheres that may be more useful in other situations.  For instance, one could argue that $\Sigma(n)$ is the $(n-1)$-skeleton of the boundary complex of a polytope.  The boundary complex of a polytope is shellable \cite{BruggesserMani71}, and rank truncation preserves shellability \cite[Theorem~4.1]{Bjorner80}, so $\Sigma(n)$ is shellable.  Note also that the puckering operation preserves shellability, so the complex $\cP$ appearing in the proof above is even shellable.  It is perhaps also worth noting that Babson and Hersh have shown that any shellable pure complex has a discrete Morse flow with one critical vertex and all other critical cells in the top dimension \cite{BabsonHersh05}, so $\Sigma(n)$ and $\cP$ also have such a flow.
\end{remark}

\section{Weights from singularities}  \label{singularities section}

Here we consider $j$th graded pieces of the weight filtration for small $j$. The groups $\Gr_j^W H^k(X)$ have a particularly nice description when the singular locus of $X$ is smooth or, more generally, when $X$ has a resolution of singularities with smooth discriminant.  We will be especially interested in the case where $X$ has isolated singularities.

\begin{proposition} \label{resolution weights}
Let $\pi: \widetilde X \rightarrow X$ be a proper birational morphism from a smooth variety such that the discriminant $V$ is smooth, and let $E = \pi^{-1}(V)$ be the exceptional locus.  Then there are natural isomorphisms
\[
W_j H^k(X) \cong W_j H^{k-1}(E),
\]
for $j$ less than $k-1$, and
\[
\Gr_{k-1}^W H^k(X) \cong \coker \big[ H^{k-1}(\widetilde X) \oplus H^{k-1}(V) \rightarrow \Gr_{k-1}^W H^{k-1}(E) \big]
\]
\end{proposition}

\begin{proof}
The proposition follows from the Mayer-Vietoris sequence 
\[
\cdots \rightarrow H^{k-1}(\widetilde X) \oplus H^{k-1}(V) \rightarrow H^{k-1}(E) \rightarrow H^k(X) \rightarrow H^k(\widetilde X) \oplus H^k(V) \rightarrow \cdots,
\]
since $H^k(\widetilde X)$ and $H^{k}(V)$ have weights in $\{k, \ldots, 2k\}$. 
\end{proof}

Putting together Propositions~\ref{compactification weights}, \ref{proper weights}, and \ref{resolution weights} gives a fairly complete description of the weight filtration on $H^*(X)$ when the singular locus is smooth and proper or, more generally, when $X$ has a resolution of singularities with smooth and proper discriminant, in which the low weight pieces come from the singularities, and the high weight pieces come from the boundary.

\begin{theorem} \label{smooth compact discriminant}
Let $\pi: \widetilde X \rightarrow X$ be a proper birational morphism from a smooth variety such that the discriminant $V$ is smooth and proper, and let $E$ be the exceptional locus.  Let $X^+$ be a smooth compactification of $\widetilde X$ with boundary $\partial X^+$.  Then the graded pieces of the weight filtration on $H^*(X)$ are given by
\[
\Gr_j^W H^k(X) \cong \left \{ \begin{array}{ll} \Gr_j^W H^{k-1}(E) & \mbox{ for } j \leq k-2. \\
									 \coker \big[ W_{k-1} H^{k-1}(\widetilde X) \oplus H^{k-1}(V) \rightarrow \Gr_{k-1}^W H^{k-1}(E) \big] & \mbox{ for } j = k-1. \\
									 \ker \big[ W_k H^k (\widetilde X) \oplus H^k(V) \rightarrow \Gr_k^W H^{k}(E) \big] & \mbox{ for } j = k. \\
									 \big(\coker \big[ H^{2n-k-1}(X^+) \rightarrow 
									\Gr_{2n-k-1}^W H^{2n-k-1}(\partial X^+) \big] \big)^\vee & \mbox{ for }j = k + 1. \\
									 \big( \Gr^W_{2n-j}H^{2n-k-1}(\partial X^+)\big)^\vee & \mbox{ for } j \geq k + 2. \end{array} \right.
\]
\end{theorem}

\begin{proof}
The first two isomorphisms follow from Proposition~\ref{resolution weights}, and the last two follow from Propositions~\ref{compactification weights} and \ref{proper weights}.  The remaining isomorphism, when $j = k$, follows from the Mayer-Vietoris sequence
\[
\cdots \rightarrow H^{k-1}(E) \rightarrow H^k(X) \rightarrow H^k(\widetilde X) \oplus H^k(V) \rightarrow H^k(E) \rightarrow \cdots,
\] 
because $\Gr_k^W H^{k-1}(E)$ vanishes, $H^k(\widetilde X)$ has weights greater than or equal to $k$, and $H^k(V)$ has pure weight $k$.
\end{proof}

\noindent Note that if the resolution and compactification are chosen so that the exceptional locus $E$ and the boundary $\partial X^+$ are divisors with simple normal crossings then each of the homology and cohomology groups appearing in the corollary, is either a pure Hodge structure of the form  $H^k(X^+)$ or $W_k H^k(\widetilde X)$, or can be expressed combinatorially as subquotients of homology and cohomology groups of intersections of components of the boundary and exceptional divisors.

\begin{remark} \label{isolated duality}
Suppose $X$ is compact and its singular locus is contained in a smooth closed subset $V \subset X$.  Let $\pi: \widetilde X \rightarrow X$ be a log resolution with respect to $V$ and let $E = \pi^{-1}(V)$ be the exceptional divisor.  Then $W_j H^k(X)$ and $W_j H^k_c(X \smallsetminus V)$ are both isomorphic to $W_j H^{k-1}(E)$ for $j$ less than $k-1$, and it follows that the Poincar\'e pairing on $X \smallsetminus V$ induces a perfect pairing
\[
\Gr_j^WH^k(X) \times \Gr_{2n-j}^W H^{2n-k}(X \smallsetminus V) \rightarrow \Q
\]
This justifies in many special cases, including for varieties with isolated singularities, the rough idea that the low weight pieces of the cohomology of a singular variety are dual to the high weight pieces of the cohomology of a smooth open variety.  The Mayer-Vietoris argument in the proof of Proposition~\ref{resolution weights} also shows how the situation is more complicated when the discriminant is singular, since the long exact sequence for the low weight pieces of the cohomology of $X$ also involve the low weight pieces of the cohomology of the discriminant.
\end{remark}

We conclude by relating the weight filtration on the cohomology of a compact variety with an isolated singularity to the the cohomology of resolution complex.

\begin{theorem}
Let $X$ be a normal variety that is smooth away from an isolated singular point $x$, and let $\pi:\widetilde X \rightarrow X$ be a log resolution with exceptional divisor $E$ that is an isomorphism over $X \smallsetminus x$.  The reduced cohomology of the resolution complex $\Delta(E)$ is naturally isomorphic, with degree shifted by one, to the weight zero part of the reduced cohomology of $X$
\[
\widetilde H^{k-1}(\Delta(E); \Q) \cong W_0 \widetilde H^k(X)
\]
\end{theorem}

\begin{proof}
The proposition is clear when $k$ is zero, since both sides vanish.  The resolution complex $\Delta(E)$ is connected, since $X$ is normal, so when $k$ is one we must show that $W_0 H^1(X)$ vanishes.  From the Mayer-Vietoris sequence for the resolution, we have an exact sequence of mixed Hodge structures
\[
\cdots \rightarrow H^0(\widetilde X) \oplus H^0(\mathrm{pt}) \rightarrow H^0(\mathrm{E}) \rightarrow H^1(X) \rightarrow H^1(\widetilde X) \rightarrow \cdots.
\]
Since $X$ is normal, the exceptional divisor $E$ is connected, so the map to $H^0(E)$ is surjective.  It follows that $W_0 H^1(X)$ injects into $W_0 H^1(\widetilde X)$, which vanishes since $H^1(\widetilde X)$ has pure weight one.  Therefore $W_0 H^1(X)$ vanishes, too, as required.

When $k$ is at least two, we have $W_0 H^{k-1}(E) \cong W_0 H^{k}(X)$, by Proposition~\ref{resolution weights}.  Then $W_0 H^{k-1}(E)$ is isomorphic to $H^{k-1}(\Delta(E); \Q)$, since $E$ has simple normal crossings, and the theorem follows.
\end{proof}

\section{Examples of resolution complexes of singularities}  \label{resolution section}

As noted in the introduction, the resolution complex of an isolated normal Cohen-Macaulay singularity has the rational homology of a wedge of spheres, and the resolution complex of a rational singularity has the rational homology of a point.  Stepanov asked whether the resolution complex of an isolated rational singularity has the homotopy type of a point \cite{Stepanov05}.  Here we give a negative answer.

The following is an example of a rational threefold singularity whose resolution complex has the homotopy type of $\R\P^2$.  The singularity is a deformation of a cone over a degenerate Enriques surface, inspired by a suggestion of J. Koll\'ar. Note that the higher cohomology groups of the structure sheaf of an Enriques surface vanish, so the cone over an Enriques surface is a rational singularity, and the dual complex of a semistable, totally degenerate Enriques surface is a triangulation of $\R\P^2$ \cite{Morrison81}.

\begin{example} \label{Enriques}
Consider the affine cone over $\P^1 \times \P^1 \times \P^1$, embedded by $\mathcal{O}(2,2,2)$.  This is a four dimensional variety in $\C^{27}$ with an isolated singularity at the origin, and the coordinates on $\C^{27}$ are naturally labeled $x_{ijk}$, with $i$, $j$, and $k$ in $\{0,1,2\}$.  If one considers $\mathcal{O}(2,2,2)$ as a toric line bundle corresponding to twice the unit cube in $\R^3$, then $x_{ijk}$ is the torus isotypical section corresponding to the lattice point $(i,j,k)$.  In particular, $x_{111}$ is the unique lattice point in the interior of this cube and its vanishing locus is the toric boundary, the complement $D$ of the cone over $\C^* \times \C^* \times \C^*$.  A general section of $\mathcal{O}(2,2,2)$ cuts out a $K3$ surface.

Let $X$ be the threefold cut out in this cone by
\[
x_{111} + x_{000}^2 + x_{002}^2 + x_{020}^2 + x_{022}^2 + x_{200}^2 + x_{202}^2 + x_{220}^2 + x_{222}^2.
\]
To first order at the cone point, $X$ looks like the cone over the toric boundary of $\P^1 \times \P^1 \times \P^1$, which is six copies of $\P^1 \times \P^1$ glued together like the faces of a cube.  So the exceptional fiber in the blowup of $X$ at the origin is a copy of this degenerate $K3$ surface $D$, whose dual complex is the boundary of an octahedron.  In local coordinates, one can check (e.g. using Jacobian matrices and computer algebra software) that $X$ is smooth away from the cone point, and is resolved by this single blowup.

Now $X$ is invariant under the involution that takes $x_{ijk}$ to $x_{(2-i)(2-j)(2-k)}$ and does not contain any of the lines that are fixed by the involution, which are exactly the cones over the 2-torsion points in $\C^* \times \C^* \times \C^*$.  So the image $y$ of the origin is an isolated singularity of the quotient $Y$ of $X$ by this involution.  The involution lifts to the blowup of $X$ at the origin and acts freely on the exceptional fiber, inducing the antipodal map on $\Delta(D)$.  It follows that the quotient of this blowup by the involution is a resolution of $(Y,y)$, and the exceptional fiber consists of three copies of $\P^1 \times \P^1$, with an appropriate gluing such that the dual complex is a triangulation of $\R\P^2$.  A \v{C}ech computation shows that the higher cohomology groups of the structure sheaf of the exceptional fiber vanish, so $(Y,y)$ is a rational singularity.
\end{example}

We conclude with a computation of resolution complexes for normal hypersurface singularities that are nondegenerate with respect to Newton polyhedra.  Let $X$ be a hypersurface in $\C^{n+1}$ defined by an equation $f = a_1 x^{u_1} + \cdots + a_r x^{u_r}$, with $u_i$ distinct exponents in $\N^{n+1}$ and $a_i \in \C^*$.  The Newton polyhedron of $f$ is the Minkowski sum
\[
\Gamma = \conv\{u_1, \ldots, u_r \} + \R_{\geq 0}^{n+1},
\]
and the inner normal fan $\Sigma_\Gamma$ is a subdivision of the positive orthant in the dual real vector space.  The singularity $(X,0)$ is \emph{general with respect to its Newton polyhedron} if the restriction of $f$ to each face of $\Gamma$ cuts out a smooth hypersurface in $(\C^*)^{n+1}$.  The following theorem generalizes partial results of Stepanov from \cite[Section~5]{Stepanov06b}.

\begin{theorem}  \label{thm:hypersurface}
Let $(X,0)$ be a normal hypersurface singularity that is general with respect to its Newton polyhedron.  Then the resolution complex of $(X,0)$ is homotopy equivalent to a wedge sum of spheres of dimension $n-1$.
\end{theorem}

\noindent  In particular, if $W_0H^n(X)$ vanishes then the resolution complex of $(X,0)$ is contractible.

In the proof of Theorem~\ref{thm:hypersurface}, we will use the following lemma on subcomplexes of subdivisions of polytopes.

\begin{lemma} \label{interior complex}
Let $\mathcal S$ be a subdivision of a polytope $P$ of dimension $n$, and let $\mathcal S_0$ be the union of the nonmaximal faces of $\mathcal S$ that are contained in the interior of $P$.  If $\mathcal S_0$ is connected then it is homotopy equivalent to a wedge sum of spheres of dimension $n-1$. 
\end{lemma}

\begin{proof}
The lemma is trivial if $n$ is less than three, so assume $n$ is at least three.

Let $\mathcal S^{(n-1)}$ be the union of the nonmaximal cones of $\mathcal S$.  Then $\mathcal S^{(n-1)}$ is the $n-1$ skeleton of the contractible complex $\mathcal S$, and hence has the homotopy type of a wedge sum of $n-1$ dimensional spheres.  Furthermore, if we label the maximal faces of $\mathcal S$ as $F_1, \ldots, F_s$, it follows from the Mayer-Vietoris sequence for the union of $\mathcal S^{(n-1)} \cup F_1 \cup \cdots \cup F_i$ with $F_{i+1}$ that the boundaries of $F_1, \ldots, F_r$ give a basis for the integral homology group $H_{n-1}(\mathcal S^{(n-1)}, \Z)$.

Note that $\mathcal S_0' = \mathcal S^{(n-1)} \smallsetminus (\mathcal S^{(n-1)} \cap \partial P)$ deformation retracts onto $\mathcal S_0$.  To see this, choose a triangulation of $\mathcal S$ with no additional vertices, for example by ordering the vertices of $\mathcal S$ and performing a pulling triangulation.  If $F$ is a face of this triangulation, and $F_0$ is the maximal subface of $F$ contained in the interior of $P$, then $F \smallsetminus (F \cap \partial P)$ canonically deformation retracts onto $F_0$.  Performing these deformation retracts simultaneously on the faces of the induced triangulation of $\mathcal S^{(n-1)}$ proves the claim.

Let $D$ be a small neighborhood of $\partial P$ in $\mathcal S^{(n-1)}$.  Then $\mathcal S^{(n-1)}$ is the union of $D$ and $\mathcal S_0'$.  After renumbering, say the maximal faces $F_1, \ldots, F_r$ meet the boundary of $P$, and $F_{r+1}, \ldots, F_s$ do not.  An argument similar to the preceding paragraph shows that $D \cap \mathcal S_0'$ is homotopy equivalent to a wedge sum of $r-1$ spheres of dimension $n-2$, and its intersection with the boundaries of $F_1, \ldots, F_{r-1}$ give a basis for the integral homology group $H_{n-2}(D \cap \mathcal S_0', \Z)$.

Applying the Van Kampen and Mayer-Vietoris theorems to the open cover
\[
\mathcal S^{(n-1)} = D \cup \mathcal S_0',
\]
shows that the fundamental group of $\mathcal S_0'$ is trivial and the integral reduced homology of $\mathcal S_0'$ is free and concentrated in degree $n-1$.  Therefore, by the Whitehead and Hurewicz theorems, $\mathcal S_0'$ has the homotopy type of a wedge sum of $n-1$ spheres, and the lemma follows, since $\mathcal S_0'$ deformation retracts onto $\mathcal S_0$.
\end{proof}

\begin{proof}[Proof of Theorem~\ref{thm:hypersurface}]
The inner normal fan $\Sigma_\Gamma$ of the Newton polytope of $(X,0)$ is a subdivision of the positive orthant in $\R^n$ and hence corresponds to a proper birational toric morphism $X(\Sigma_\Gamma) \rightarrow \C^{n+1}$.  Because $(X,0)$ is general with respect to $\Gamma$, the strict transform $\widetilde X$ meets every torus invariant subvariety $V(\sigma)$ properly and, furthermore, $\widetilde X \cap V(\sigma)$ is connected if it is positive dimensional, and consists of $\ell$ points if $\sigma$ is the codimension one cone corresponding to a compact edge of lattice length $\ell$ in the Newton polyhedron \cite{Varchenko76}.  If $\Sigma_\Gamma$ is unimodular then this morphism gives an embedded resolution of $X$, and if every component of $\pi^{-1}(0)$ is a divisor then this is a weak log resolution.  We consider first the resolution complex in this special case.

Let $\mathcal{S}$ be the subdivision of the standard simplex $\conv \{e_0^*, \ldots, e_n^*\}$ induced by $\Sigma_\Gamma$, and let $\mathcal{S}_0$ be the union of the nonmaximal faces of $\mathcal{S}$ that are contained in the interior of the simplex.  The discussion above says that the resolution complex $\Delta(E)$ is obtained from $\mathcal{S}_0$ by taking the $\ell$-fold puckering along each $n-1$ dimensional face corresponding to a compact edge of lattice length $\ell$ in $\Gamma$.  Since $(X,0)$ is assumed to be normal, the resolution complex must be connected, and hence $\mathcal{S}_0$ is connected as well.  By Lemma \ref{interior complex}, $\mathcal{S}_0$ is homotopy equivalent to a wedge sum of spheres of dimension $n-1$ and, by Lemma~\ref{puckering}, the puckering $\Delta(E)$ is so, too.

We now extend the arguments above to the general case, where $\widetilde X \rightarrow X$ is not necessarily a log resolution with respect to zero.  Let $\Sigma'$ be a unimodular refinement of the barycentric subdivision of $\Sigma_\Gamma$.  Then $\pi': X(\Sigma') \rightarrow \C^{n+1}$ is an embedded resolution of $X$ that factors through $X(\Sigma_\Gamma)$.  Since $(X,0)$ is general with respect to $\Gamma$, the strict transform $X'$ of $X$ meets all of the boundary components of the smooth toric variety $X(\Sigma')$ in the preimage of 0 transversely, and the barycentric subdivision ensures that every component of the preimage $E'$ of 0 is a divisor.  Since the toric boundary of $X'$ is a divisor with simple normal crossings, it follows that $E'$ is a divisor with simple normal crossings as well.  In particular, the restriction of $\pi'$ to the strict transform of $X$ is a weak log resolution of $X$ with respect to 0.

Let $\mathcal{S}'$ be the union of the faces in the subdivision of the unit simplex induced by $\Sigma'$ that are contained in nonmaximal faces of $\mathcal S$, and let $\mathcal P \mathcal S'$ be constructed by taking the $\ell$-fold puckering of each $(n-1)$-dimensional face corresponding to a compact edge of length $\ell$ in the Newton polyhedron.  As in the proof of Lemma~\ref{interior complex}, the subcomplex $\mathcal P \mathcal S'_0$ consisting of faces entirely contained in the interior of the simplex is homotopy equivalent to $\mathcal \P \mathcal S'$ minus the boundary of the simplex,  the full complex $\mathcal P \mathcal S'$ is homotopy equivalent to a wedge sum of spheres, and applications of the Van Kampen, Mayer-Vietoris, and Whitehead and Hurewicz Theorems show that $\mathcal S_0'$ is homotopy equivalent to a wedge sum of spheres.

Now $E'$ has one vertex for each ray of $\Sigma'$ in the interior of the positive orthant that is contained in a face of codimension at least two in $\Sigma_\Gamma$, and $\ell$ vertices for each ray that is contained in a codimension one cone of $\Sigma_\Gamma$ corresponding to a compact of length $\ell$ in the Newton polyhedron.  By considering how the resolution factors through $X(\Sigma_\Gamma)$, one sees that the resolution complex $\Delta(E')$ is naturally identified with $\mathcal P \mathcal S'_0$, and the theorem follows.
\end{proof}

\begin{remark}
The embedded resolutions of isolated hypersurface singularities induced by subdivisions of the positive orthant are not in general isomorphisms away from 0.  Therefore, it is essential in the proof Theorem~\ref{thm:hypersurface} that the resolution complex does not depend on the discriminant of the resolution.
\end{remark}

\begin{remark}
The proof of Theorem~\ref{thm:hypersurface} gives a combinatorial formula for the number of spheres in the resolution complex of $(X,0)$, which is also the dimension of $W_0H^n(X)$.  It is exactly the number of vertices of the Newton polyhedron that are contained in no unbounded facets plus the sum over all compact edges of the Newton polyhedron of the lattice length minus one.
\end{remark}

\begin{remark}
Similar methods can be used to compute resolution complexes for some normal complete intersection singularities that are general with respect to Newton polyhedra.  Suppose $\{X_1, \ldots, X_r \}$ is a collection of hypersurfaces that is generic with respect to Newton polyhedra $\Gamma_1, \ldots, \Gamma_r$, and let $X$ be the complete intersection $X_1 \cap \cdots \cap X_r$.  If the singularity $(X,0)$ is normal and the Newton polyhedra all have the same normal fan, then an argument similar to the one above for hypersurfaces shows that the resolution complex of $(X,0)$ has the homotopy type of a wedge sum of spheres of dimension $\dim X - 1$.
\end{remark}

\begin{corollary}
The resolution complex of an isolated rational singularity that is general with respect to its Newton polyhedron is contractible.
\end{corollary}
 
\bibliography{math}
\bibliographystyle{amsalpha}

\end{document}